\documentclass[reqno,11pt]{amsart}
\usepackage{tikz}
\usetikzlibrary{arrows.meta}
\usetikzlibrary{decorations.pathmorphing}
\usepackage{esint}
\usepackage{verbatim}
\usepackage{graphicx}
\usepackage{subfigure}
\usepackage{amsfonts}
\usepackage{amssymb}
\usepackage{amsmath}
\usepackage{amsthm}
\usepackage{latexsym}
\usepackage{amscd}
\usepackage{mathrsfs}
\usepackage{verbatim}
\usepackage{mathrsfs}
\usepackage{enumitem}
\usepackage{braket}
\usepackage{float}

\usepackage[margin=1.3in]{geometry}
\usepackage{comment}
\usepackage{bm}
\usepackage{xcolor}
\usepackage{imakeidx}
\makeindex[title=Index of Notation]

\usepackage{hyperref}
\hypersetup{
	colorlinks=true,
	linkcolor=black,
	 urlcolor=black,
	citecolor=black
}
\usepackage{caption}
\usepackage{array}

\usepackage{amssymb}
\usepackage{pdfpages}
\usepackage{verbatim}

\usepackage{appendix}
\usepackage{xspace}

\newtheorem{theorem}{Theorem}[section]
\newtheorem{conjecture}[theorem]{Conjecture}
\newtheorem{lemma}[theorem]{Lemma}
\newtheorem{proposition}[theorem]{Proposition}
\newtheorem{corollary}[theorem]{Corollary}
\theoremstyle{definition}
\newtheorem{definition}[theorem]{Definition}

\theoremstyle{remark}
\newtheorem{remark}[theorem]{Remark}

\newcommand{\R}{\mathbb{R}}

\newcommand{\Rbar}{\overline{\mathrm{Rm}}}
\newcommand{\inj}{\operatorname{inj}}
\newcommand{\spt}{\operatorname{spt}}
\newcommand{\graph}{\operatorname{graph}}
\newcommand{\tr}{\operatorname{tr}}
\newcommand{\Div}{\operatorname{div}}
\newcommand{\N}{\mathcal N}
\newcommand{\X}{\mathcal X}
\newcommand{\Y}{\mathcal Y}

\newcommand{\ip}[2]{\left\langle #1,#2\right\rangle}
\numberwithin{equation}{section}

\makeatletter
\@namedef{subjclassname@2020}{\textup{2020} Mathematics Subject Classification}
\makeatother


\begin{document}
	\title[Mean Curvature Flow]{ Singularities of Mean Curvature Flow with Bounded Total Curvature I: Tangent Flow}
	
	\author{Yongheng Han}
	\address{Institute of Mathematics, Academy of Mathematics and Systems Science, Chinese Academy of Sciences, Beijing 100190, China}
	
	\email{hyh98@amss.ac.cn}

    \author{Zhen Wang}
	\address{The Institute of Geometry and Physics, University of Science and Technology
of China, Hefei,  Anhui 230026, China}
	
	\email{wz1854@ustc.edu.cn}

	\date{\today}

	
	\begin{abstract}
    We prove that every tangent flow at the first singular time of a smooth mean curvature flow of closed embedded hypersurfaces $M_t^n\subset\R^{n+1}$, $n\geq2$, has smooth embedded support, provided that
    $\sup_{t\in[0,T)}\int_{M_t}|A|^n\,d\mu_t<\infty$.
    This extends Ilmanen's regularity result for tangent flows to higher dimensions. 
	\end{abstract}
	
\maketitle
\tableofcontents
\section{Introduction}\label{sec1}

Let \(\mathbf{x}_0:M^n\to\mathbb{R}^{n+1}\) be a smooth
embedding of a closed hypersurface. A one-parameter family
of immersions
\(\mathbf{x}:M\times[0,T)\to\mathbb{R}^{n+1}\)
is called a mean curvature flow if
\begin{equation}\label{eq0.1}
    \partial_t\mathbf{x}=\mathbf{H},
    \qquad
    \mathbf{x}(\cdot,0)=\mathbf{x}_0,
\end{equation}
where \(\mathbf{H}\) denotes the mean curvature vector of
\(M_t:=\mathbf{x}(M,t)\). Huisken \cite{MR772132} proved
that, at a finite first singular time, the second
fundamental form becomes unbounded. Understanding the
geometry of the resulting singularities is a central
problem in mean curvature flow.

Tangent flows describe singularities through parabolic
rescaling. Although these limits satisfy the self-shrinker
equation in a weak sense, their supports need not be smooth
a priori. In this paper, we study this regularity problem
under a uniform bound for the scale-invariant total
curvature \(\int_{M_t}|A|^n\,d\mu_t\). This assumption
permits curvature concentration, but bounds its total
amount and imposes strong restrictions on possible
blow-up limits.

In dimension two, finite total curvature has played an
important role in the analysis of singularities of mean
curvature flow. In particular, Ilmanen
\cite{ilmanen2026singularitiesmeancurvatureflow}
established smoothness of the support of tangent flows at
the first singular time. Our main result extends this
conclusion to higher-dimensional hypersurfaces under a
uniform total curvature bound.

\begin{theorem}
\label{thm:intro-tangent-flow}
    Let
    \(\mathbf{x}:M^n\times[0,T)\to\mathbb{R}^{n+1}\)
    be a smooth mean curvature flow of closed embedded
    hypersurfaces, where $T<\infty$ is the first singular 
    time. If
    \begin{equation}\label{eq0.2}
        \sup_{t\in[0,T)}
        \int_{M_t}|A|^n\,d\mu_t<\infty,
    \end{equation}
    then the support of every tangent flow at time \(T\)
    is smooth.
\end{theorem}

\begin{remark}
    Nakauchi \cite{MR1202157} asserted that, under a
    uniform total curvature bound, the mean curvature
    flow converges smoothly to a limiting hypersurface
    \(M_T\) away from a finite singular set
    \(\mathcal{S}\). However, the justification of
    equation~(26) in \cite{MR1202157} appears to require
    an additional argument. We do not rely on that
    conclusion.
\end{remark}

Theorem~\ref{thm:intro-tangent-flow} concerns regularity of
the support and does not determine the multiplicity of the
tangent flow. Ding \cite{MR3864377} proved that a complete 
proper self-shrinker with finite total curvature is either
compact or a hyperplane. Thus, in our setting, the remaining
questions concern multiplicity and the exclusion of planar
tangent flows at singular points. Motivated by the
multiplicity-one theorem of Bamler--Kleiner
\cite{bamler2024multiplicityconjecturemeancurvature}
for surface flows in \(\mathbb{R}^3\), we propose the
following conjecture.

\begin{conjecture}\label{conj:intro-multiplicity-one}
    Let
    \(\mathbf{x}:M^n\times[0,T)\to\mathbb{R}^{n+1}\)
    be a smooth mean curvature flow of closed embedded
    hypersurfaces satisfying \eqref{eq0.2}, where \(T\)
    is the first singular time. Then the rescaled mean
    curvature flow converges smoothly, with multiplicity
    one, to a smooth closed self-shrinker.
\end{conjecture}

\begin{remark}
   In a forthcoming paper, building on the results of the present work, we will provide a proof of this conjecture \cite{HW26b}.
\end{remark}
Our proof follows the broad strategy of Ilmanen
\cite{ilmanen2026singularitiesmeancurvatureflow}:
identify the curvature concentration set, establish
regularity away from it, and remove the remaining isolated
singularities. The implementation in higher dimensions
requires different tools at both regularity steps.

In dimension two, Ilmanen uses a localized Gauss--Bonnet
argument to obtain scale-invariant local
\(L^2\)-curvature bounds for suitable rescaled slices.
Away from the resulting concentration points, the local
curvature energy is small on appropriate scales.
%Simon's graphical decomposition \cite{MR1243525} then decomposes the rescaled surfaces into finitely many nearly flat graphs. 
Simon's graphical decomposition \cite{MR1243525} then provides the local layer structure and the area-ratio control.
Passing to the limit, Allard's regularity
theorem and elliptic bootstrapping for the self-shrinker
equation yield smoothness away from isolated points.
Near each remaining concentration point, Ilmanen identifies
annular ends with finite area and bounded mean curvature.
Gulliver's removable singularity theorem \cite{MR431045}
extends these ends across the puncture as branched
immersions. In ambient dimension three, embeddedness and
the strong maximum principle yield smoothness of the
support across the concentration points.

In higher dimensions, the assumed \(L^n\)-curvature bound \eqref{eq0.2} supplies the required control of curvature concentration. 
The first difficulty is to obtain regularity where the local curvature energy is small. 
%Small \(L^n\)-curvature alone does not directly supply the area-ratio control needed to apply Allard's regularity theorem. 
Rather than extending Simon's graphical decomposition itself, we derive the componentwise volume estimates needed to apply Allard's regularity theorem.
%We address this issue by constructing a local mean curvature flow and using its estimates to control the volume of each local connected component. 
These estimates are obtained by evolving the initial submanifold for a short time by mean curvature flow and comparing the evolved geometry with the initial one.
This result is also of independent interest and holds in arbitrary codimension.

\begin{proposition}[Theorem~\ref{thm:volume ratio by small total curvature}, Corollary~\ref{cor:almost Euclidean volume ratio of local component}]
\label{prop:intro-volume}
    For every \(\nu>0\), there exist constants
    \(\epsilon=\epsilon(n,m,\nu)>0\) and
    \(C=C(n,m)>0\) with the following property.
    Let \(\rho>0\), and let \(M^n\) be a smooth
    submanifold properly embedded in
    \(B_\rho\subset\mathbb{R}^{n+m}\).
    Suppose that
    \[
        \|A\|_{L^n(B_\rho\cap M)}
        \leq\delta\leq\epsilon
    \]
    and
    \[
        \sup_{\substack{x\in B_{5\rho/6}\cap M\\
                        0<s\leq\rho/6}}
        s^{-n}\mu(B_s(x)\cap M)\leq\nu.
    \]
    Then, for every \(x\in B_{\rho/4}\cap M\) and
    \(0<r\leq\rho/512\),
    \[
        (1-C\delta^{4/3})\omega_n r^n
        \leq
        \mu(V_{x,r})
        \leq
        \mu(E_{x,r})
        \leq
        (1+C\delta^{4/3})\omega_n r^n,
    \]
    where 
    \[
    V_{x,r}:=C_x(B_r(x)\cap M),\qquad
    E_{x,r}:=B_r(x)\cap C_x(B_{2r}(x)\cap M).
    \]
\end{proposition}

The proof uses mean curvature flow to relate small total
curvature to almost Euclidean volume. We first construct
a local flow whose existence time is independent of the
initial pointwise curvature bound. 
After scaling to $\rho=1$, local curvature-energy estimates and an iterative continuation argument yield the improved estimate
\[
    |A|^2\leq C\delta^2t^{-1},\quad t\in(0,C|\log\delta|^{-1}],
\]
in a smaller ambient ball. At a positive time
\(t_*=\theta r^2\) with \(r\in(0,1)\), this gives
\[
    r^2|A|^2\leq C\delta^2\theta^{-1}.
\]
At this time, we can represent a suitable neighborhood
of the evolved connected component as a graph and
estimate its volume ratio. Meanwhile, each point moves
a distance of at most \(C\delta\sqrt{\theta}\,r\) under
the flow. We compare the evolved component with graphical
regions in slightly smaller and larger balls, and use
the local curvature energy to control the volume loss.
This allows us to estimate the volume of the original
connected component.
The  error is bounded by
\[
    C(n,m)\bigl(
        \delta^2\theta^{-1}
        +\delta\sqrt{\theta}
    \bigr)
    +C(n,m,\nu)\delta^2\theta.
\]
Choosing \(\theta=\delta^{2/3}\) balances the first two
terms, each of which is of order \(\delta^{4/3}\).
The last term is of order \(\delta^{8/3}\) and can be
absorbed by taking \(\delta\) sufficiently small depending
on \(n\), \(m\), and \(\nu\). This gives the error
\(C(n,m)\delta^{4/3}\) in Proposition~\ref{prop:intro-volume}.
The componentwise volume estimate and Allard’s regularity theorem yield smooth limiting sheets away from finitely many concentration points. 
A one-sidedness argument and the strong maximum principle then show that sheets meeting at a point coincide.

\iffalse
These estimates
allow us to construct a global parametrization by
\(\mathbb{R}^n\), giving the following result.

\begin{theorem}[Topological rigidity]
\label{thm:intro-rigidity}
    For every \(\nu>0\), there exists
    \(\epsilon=\epsilon(n,m,\nu)>0\) such that the
    following holds. Let
    \(\mathbf{x}:M^n\to\mathbb{R}^{n+m}\) be a smooth
    proper immersion of a connected manifold.
    Suppose that
    \[
        \sup_{\substack{y\in\mathbf{x}(M)\\r>0}}
        r^{-n}
        \mu\bigl(\mathbf{x}^{-1}(B_r(y))\bigr)
        \leq\nu
    \]
    and
    \[
        \left(\int_M|A|^n\,d\mu\right)^{1/n}
        \leq\epsilon,
    \]
    where \(\mu\) is the measure induced by the
    immersion. Then \(M\) is diffeomorphic to
    \(\mathbb{R}^n\).
\end{theorem}

\begin{remark} 
For $n=2$, if the total curvature is sufficiently small, by the work of Huber \cite{MR94452} and Müller-Šverák \cite{MS95}, $M$ is diffeomorphic to $\mathbb{R}^n$.
If $M$ has bounded second fundamental form and small total curvature, the first author \cite{MR4881490} showed that the mean curvature flow with initial data $M$ exists for all time. 
Related topological rigidity results under intrinsic Sobolev and volume-growth assumptions were recently obtained by Lee and Lee \cite{lee2026localmollificationmetricssmall} using Ricci-flow smoothing.  We also conjecture that the theorem remains valid without the volume ratio assumption.
\end{remark}

\fi

The second difficulty is to remove these isolated singularities without relying on a two-dimensional removable singularity theorem. 
Lower semicontinuity of the curvature energy, together with local graphical estimates, yields decay of the scale-invariant curvature and its derivatives near each concentration point.
Consequently, every tangent cone there is an integer-multiplicity hyperplane.
Using the curvature decay estimates, we establish uniqueness of this tangent hyperplane, which allows us to express the punctured support as finitely many graphs over a common plane.
These graphs extend smoothly across the puncture, and the strong maximum principle shows that the ordered sheets coincide.
This proves smoothness of the support at every concentration point and completes the proof
of Theorem~\ref{thm:intro-tangent-flow}.

\medskip
\noindent\textbf{Organization of the paper.}
Section~\ref{sec2} collects preliminary results on mean
curvature flow and varifolds.
Section~\ref{sec3} establishes local curvature-energy
estimates and improved curvature bounds.
Section~\ref{sec4} constructs local mean curvature flows
and proves Proposition~\ref{prop:intro-volume}.
%%Section~\ref{sec5} obtains immortal flows and proves
%%Theorem~\ref{thm:intro-rigidity}.
Section~\ref{sec6} proves regularity of tangent-flow
supports. %The two appendices provide the short-time
%existence theorem and the uniqueness of tangent planes
%used in the proofs.

$\\$
\textbf{Notation and conventions.}
\begin{itemize}
    \item $n\geq2$ and $m\geq1$ denote the dimension and codimension, respectively.
    \item $C$ denotes a positive constant which may vary from line to line, with its dependence indicated explicitly whenever necessary.
    \item $\omega_n$ denotes the volume of the unit ball in $\R^n$.
    \item $g_e$ denotes the Euclidean metric of $\R^{n+m}$.
    \item $\mathbf x$ denotes the immersion, also viewed as the position vector; along a flow, we write $\mathbf x_t=\mathbf x(\cdot,t)$.
    \item $M=M^n$ denotes the parameter manifold, and $M_t:=\mathbf x_t(M)$ denotes the evolving submanifold which retains the immersion data.
    \item $g=\mathbf x^* g_e$ and $\mu=\mu_{g}$ denote the induced metric and volume measure, written as $g(t)$ and $\mu_t$ along the flow.
    \item The superscripts $\top$ and $\perp$ denote the tangential and normal projections, respectively.
    \item $\nabla$ denotes the induced covariant derivative on the relevant tensor bundle.
    \item $\Delta:=\tr\nabla^2$ and $\square:=\partial_t-\Delta$.
    \item $A$ and $\mathbf H$ denote the second fundamental form and mean curvature vector, respectively, with $\mathbf H=\tr_g A$.
    \item $p_t:=\mathbf x_t(p)$ denotes the trajectory for $p\in M$ along the flow.
    \item $\Phi_{t,s}$ denotes the evolution map along trajectories, defined by $\Phi_{t,s}(p_t)=p_s$ whenever the corresponding trajectory exists throughout the interval between $t$ and $s$.
    \item Along the flow, geometric quantities and operators are understood at the time under consideration, with their time dependence suppressed when no confusion arises.
    \item $B_r(x)$ denotes the Euclidean ball of radius $r$ centered at $x$; we write $B_r=B_r(0)$.
    \item $A_{r,R}(x):=B_R(x)\setminus \overline{B}_r(x)$ for $0<r<R$; we write $A_{r,R}=A_{r,R}(0)$.
    \item $B_r^P(x):=B_r(x)\cap P$, $A_{r,R}^P(x):=A_{r,R}(x)\cap P$, for an affine $n$-plane $P\subset\R^{n+m}$ passing through $x$.
    \item $(U)_r:=\{\,x\in U\mid B_r(x)\Subset U\,\}$ for an open set $U\subset\R^{n+m}$ and $r>0$.
    \item $C_x(U)$ denotes the connected component of $U$ containing $x$ for $x\in U$.
    \item $a\wedge b:=\min\{a,b\}$ for $a,b\in\R$.
\end{itemize}

$\\$
\textbf{Disclosure of AI use.}  
The main results and proof strategies were developed by the authors independently of AI in the spring of 2026. 
During the subsequent preparation and revision of the manuscript, the authors used ChatGPT (OpenAI, GPT-6-Astra) to assist with drafting the appendix, checking and refining mathematical arguments, and improving the language, organization, notation, and LaTeX presentation. 
All AI-assisted revisions incorporated into the manuscript were independently checked by the authors, who take full responsibility for its content.

$\\$
\noindent
\textbf{Acknowledgment.}
The authors would like to thank Hanbing Fang and Haozhao Li for helpful discussions. 
The authors would like to thank Manchun Lee for clarifying several details of the paper \cite{lee2026localmollificationmetricssmall}.
The authors are supported by the Project of Stable Support for Youth Team in Basic Research Field, Chinese Academy of Sciences (YSBR-001).

\section{Preliminaries}\label{sec2}

\subsection{Local short-time existence of mean curvature flow}
$\\$

\begin{definition}
Let $\mathbf x:M^n\to\mathbb{R}^{n+m}$ be a smooth immersion (or embedding), where $M$ may have boundary, and let $\Omega\subset\mathbb{R}^{n+m}$ be open.
We say that $\mathbf x$ is properly immersed (or embedded) and without boundary
in $\Omega$ if
\[
\mathbf x^{-1}(\Omega)\cap\partial M=\emptyset
\]
and the restriction
\[
\mathbf x|_{\mathbf x^{-1}(\Omega)}:\mathbf x^{-1}(\Omega)\to\Omega
\]
is proper.
Throughout this paper, ``properly immersed (or embedded) in $\Omega$'' or ``a proper immersion (or embedding) in $\Omega$'' will always include the requirement that there is no boundary in $\Omega$.
\end{definition}

When a mean curvature flow is restricted to an open subset of the ambient space, its parameter domain may depend on time. 
We therefore use the following convention.

\begin{definition}
A smooth mean curvature flow \(\{M_t\}_{t\in[0,T]}\) in an open set \(\Omega\subset\mathbb R^{n+m}\) is a family of immersed submanifolds satisfying the mean curvature flow equation \eqref{eq0.1} locally in space and time.
The notation \(M_t\) retains the immersion data, and all integrals are taken with respect to the induced measure, counting multiplicities. 
Empty time slices are allowed.
A trajectory can terminate at an interior time only by leaving every compact subset of the ambient domain.
\end{definition}

The metric on a local region can be made complete and of bounded geometry by performing a conformal transformation on the ambient space, while leaving the interior unchanged.

\begin{lemma}[Lemma~2.4 of \cite{MR4836053}, Lemma~2.1 of \cite{cheng2026meancurvatureflowapproach}]\label{lem:completion}
    There exist constants $C_1=C_1(n,m)$, $C_2=C_2(n,m)$ and $c=c(n,m)$ such that the following holds.
    Let $\mathbf x:M^n\to\R^{n+m}$ be a smooth proper immersion in an open set $V\subset\R^{n+m}$ such that 
    \[
    %(V)_1\cap M\neq\emptyset,\quad
    \sup_{V\cap M}|A|\leq 1.
    \]
    Then, for any $l>C_1$, there exist an open subset $\tilde{V}\subset V$ and a metric $h$ defined on $\tilde{V}$ such that 
    \begin{enumerate}[label=(\roman*)]
        \item 
        $(\tilde{V},h)$ is complete with $\inj_{h}(\tilde{V})\geq cl^{-\frac12}$ and $\sup_{\tilde{V}}|\nabla^j\mathrm{Rm}|_{h}\leq \Lambda_j(n,m)l^{1+j/2}$ for each integer $j\geq0$;
        \item 
        $\mathbf x:\tilde{V}\cap M \to (\tilde{V},h)$ is a complete submanifold with 
        $\sup_{\tilde{V}\cap M}|A|_h\leq l$;
        \item
        $(V)_1\subset\tilde{V}$ and $h\equiv g_{e}$ on $(\tilde{V})_{C_2l^{-1/2}}$.
    \end{enumerate}
\end{lemma}

\iffalse
\begin{lemma}\label{lem:general short-time existence}
    Let $(\bar{M}^{n+m},\bar{g})$ be a complete manifold with bounded geometry and  $M^n\subset\bar{M}^{n+m}$ be a complete submanifold.
    Suppose that 
    \[
    \sup_{M}|A|^2\leq \rho^{-2}
    \]
    for some $\rho>0$.
    Then there exists a smooth mean curvature flow $\{M_t^n\}_{t\in[0,\gamma\rho^2]}$ in $\bar{M}$ with initial data $M_0=M$ such that
    \[
    \sup_{t\in[0,\gamma\rho^2]}\sup_{M_t}|A|^2\leq \lambda\rho^{-2}.
    \]
    Here, $\gamma=\gamma(n,m)$, and $\lambda$ depends on $n$, $m$ and the geometry of $(\bar{M},\bar{g})$.
\end{lemma}
\begin{proof}
    For the isometric immersion $\mathbf x_0:M\to(\bar{M},\bar{g})$ with $|A|^2\leq\rho^{-2}$, consider the rescaling 
    $\mathbf x_0: M\to (\bar{M},\rho^{-2}\bar{g})$ with $|A|^2\leq1$.
    By Theorem~\ref{thm:main}, there exists a mean curvature flow $\mathbf x:M\times[0,T]\to (\bar{M},\rho^{-2}\bar{g})$ with initial data $\mathbf x(0)=\mathbf x_0$ and $|A|^2(\cdot,t)\leq C_0$, where $T=T(n,m)$ and $C_0=C_0(n,m,i_0,\Lambda_0,\Lambda_1)$.
    Rescaling back yields a mean curvature flow $\mathbf x:M\times[0,T\rho^2]\to (\bar{M},\bar{g})$ with initial data $\mathbf x_0$ and $|A|^2(\cdot,t)\leq C_0\rho^{-2}$.
\end{proof}
\fi

Combining the completion trick from Lemma~\ref{lem:completion} and the short-time existence \cite{Han26}, %Theorem~\ref{thm:main}
 we obtain the following local short-time existence. 

\begin{proposition}[Local existence under a curvature bound]\label{prop:local short-time existence}
    There exist constants $\Gamma=\Gamma(n,m)$ and $\Lambda=\Lambda(n,m)$ such that the following holds.
    Let $M_0^n\subset\R^{n+m}$ be a submanifold properly immersed in $B_R\subset\R^{n+m}$ such that
    \[
    %B_{R-\rho}\cap M_0\neq\emptyset,\quad 
    \sup_{B_R\cap M_0}|A|^2\leq \rho^{-2}
    \]
    for some $\rho>0$ and $R>2\rho$.
    Then there exists a smooth mean curvature flow $\{M_t^n\}_{t\in[0,\Gamma\rho^2]}$ properly immersed in $B_{R-2\rho}$ with initial data $B_{R-2\rho}\cap M_0$ such that 
    \[
    \sup_{t\in[0,\Gamma\rho^2]}\sup_{B_{R-2\rho}\cap M_t}|A|^2\leq \Lambda\rho^{-2}.
    \]
\end{proposition}
\begin{proof}
    By parabolic scaling, it suffices to prove the case $\rho=1$. 
    Choose $l_0=l_0(n,m)>C_1$ sufficiently large that $C_2l_0^{-\frac12}<1$.
    For the isometric immersion $\mathbf x_0:M_0\to\R^{n+m}$, by Lemma~\ref{lem:completion} there exist an open subset $B_{R-1}\subset U\subset B_R$ and a metric $h$ on $U$ such that $(U,h)$ is a complete manifold and 
    \[
    \inj_{h}(U)\geq i_0(n,m)>0,\qquad
    \sup_U|\nabla^j\mathrm{Rm}(h)|_{h}\leq K_j(n,m),\quad
    \forall\,j\geq0. 
    \]
    Moreover, setting 
    \[
    \tilde{M}_0:=\mathbf x_0^{-1}(U),
    \]
    the induced metric $\mathbf x_0^* h$ on $\tilde{M}_0$ is complete, and the restricted immersion satisfies 
    \[
    \sup_{\tilde{M}_0}|A|_h^2\leq l_0^2.
    \]
    The completion lemma also gives $h\equiv g_e$ on $(U)_1$, which implies $h\equiv g_e$ on $B_{R-2}$.
    By short-time existence \cite{Han26}, % Theorem~\ref{thm:main}
    there exists a smooth mean curvature flow 
    \[
    F:\tilde{M}_0\times[0,\Gamma]\to(U,h)
    \]
    with initial data $F(\cdot,0)=\mathbf x_0|_{\tilde{M}_0}$ such that 
    \begin{equation}\label{ineq:local short-time existence-1}
    \sup_{t\in[0,\Gamma]}\sup_{\tilde{M}_0}|A(\cdot,t)|_h^2\leq \Lambda.
    \end{equation}
    Here, $\Gamma=\Gamma(n,m)$, $\Lambda=\Lambda(n,m)$.
    Restricting this flow to $B_{R-2}$, where $h\equiv g_e$, yields a smooth Euclidean mean curvature flow with initial data $B_{R-2}\cap M_0$ and 
    \[
    \sup_{t\in[0,\Gamma]}\sup_{F_t(\tilde{M}_0)\cap B_{R-2}}|A(\cdot,t)|_{g_e}^2\leq \Lambda.
    \]
    It remains to check the properness.
    By \eqref{ineq:local short-time existence-1} we have
    \begin{equation}\label{ineq:local short-time existence-2}
    d_h(F_t(p),F_0(p))\leq \sqrt{n\Lambda}\Gamma=:D,
    \quad \forall\,p\in\tilde{M}_0,
    \quad \forall\,t\in[0,\Gamma].
    \end{equation}
    For any compact set $K\Subset B_{R-2}$, since $(U,h)$ is complete, its neighborhood 
    \[
    K_{D}:=\{\,z\in U\mid d_h(z,K)\leq D\,\}
    \]
    is compact.
    Since $F_0=\mathbf x_0|_{\tilde{M}_0}$ is proper, 
    \[
    F_0^{-1}(K_{D})\subset \tilde{M}_0
    \]
    is compact.
    By \eqref{ineq:local short-time existence-2}, for any $t\in[0,\Gamma]$, 
    \[
    F_t^{-1}(K)\subset F_0^{-1}(K_D) 
    \]
    is compact.
    Since 
    \[
    \tilde{M}_0\subset\mathbf x_0^{-1}(B_R)\subset M_0^\circ
    \]
    has no boundary, $F_t^{-1}(B_{R-2})\subset \tilde{M}_0$ has no boundary either.
    Hence, for any $t\in[0,\Gamma]$, 
    \[
    F_t:F_t^{-1}(B_{R-2})\to B_{R-2}
    \]
    is a proper immersion in $B_{R-2}$.
    This finishes the proof.
\end{proof}

\subsection{Displacement estimates}

\begin{lemma}\label{lem:distance distortion}
    Let $\mathbf x:M^n\times[0,T]\to\R^{n+m}$ be a smooth mean curvature flow with $p,q\in M$.
    If 
    \[
    |A|^2(p,t)\leq f(t), \quad |A|^2(q,t)\leq g(t), \quad\forall\,t\in[0,T], 
    \]
    then for any $[a,b]\subset[0,T]$, 
    \begin{gather*}
        \big||\mathbf x(p,a)-\mathbf x(q,a)|-|\mathbf x(p,b)-\mathbf x(q,b)|\big|\leq \sqrt n\int_a^b(\sqrt f+\sqrt g),\\
        |\mathbf x(p,a)-\mathbf x(p,b)|\leq\sqrt n\int_a^b\sqrt f.
    \end{gather*}
    In particular, if 
    \[
    |A|^2(p,t)\leq \alpha t^{-1},
    \]
    then
    \[
    |\mathbf x(p,a)-\mathbf x(p,b)|\leq 2\sqrt{n\alpha b}-2\sqrt{n\alpha a}.
    \]
\end{lemma}
\begin{proof}
    Set $Y(t)=\mathbf x(p,t)-\mathbf x(q,t)$ for $t\in[0,T]$.
    Then
    \[
    |\partial_t Y|
    =|\mathbf H(p,t)-\mathbf H(q,t)|
    \leq \sqrt n\bigl(\sqrt {f(t)}+\sqrt {g(t)}\bigr).
    \]
    Integrating over $[a,b]$ gives the first estimate.
    The displacement estimate follows similarly.
\end{proof}

\begin{lemma}\label{lem:shrinking ball}
Fix $x_0\in\R^{n+m}$.
Let $\{M_t^n\}_{t\in[0,T]}$ be a smooth mean curvature flow in $\R^{n+m}$ satisfying
\[
\sup_{B_r(x_0)\cap M_t}|A|^2\leq \alpha t^{-1}.
\]
For $p_0=\mathbf x_0(p)\in B_r(x_0)\cap M_0$, 
if $B_{\sigma}(p_0)\subset B_r(x_0)$ with $\sigma\geq2\sqrt{n\alpha T}$, 
then for the forward trajectory $\{p_t=\mathbf x_t(p)\}_{t\geq0}$ we have
\[
B_{\sigma-2\sqrt{n\alpha t}}(p_t)
\subset B_\sigma(p_0),
\quad\forall\,t\in[0,T].
\]
For $q_T=\mathbf x_T(q)\in B_r(x_0)\cap M_T$, if
\(
B_{\rho}(q_T)\subset B_r(x_0)
\)
with $\rho\geq 2\sqrt{n\alpha T}$,
then for the backward trajectory $\{q_t=\mathbf x_t(q)\}_{t\leq T}$ we have 
\[
B_{\rho-(2\sqrt{n\alpha T}-2\sqrt{n\alpha t})}(q_t)
\subset B_\rho(q_T),
\quad\forall\,t\in[0,T].
\]
\end{lemma}
\begin{proof}
    For the first statement, define 
    \[
    \bar{T}=\sup\{\,S\in[0,T]\mid p_t\in B_r(x_0)\cap M_t,\,\forall\,t\in[0,S]\,\}.
    \]
    For $t\in[0,\bar{T}]$, by Lemma~\ref{lem:distance distortion}, 
    \[
    |p_t-p_0|\leq 2\sqrt{n\alpha t},
    \]
    and thus
    \[
    B_{\sigma-2\sqrt{n\alpha t}}(p_t)
    \subset B_\sigma(p_0)
    \]
    if $2\sqrt{n\alpha t}\leq \sigma$.
    Suppose that $\bar{T}<T$.
    By continuity and the definition of $\bar{T}$, $p_{\bar{T}}\in\partial B_r(x_0)$.
    Since $B_\sigma(p_0)\subset B_r(x_0)$, we have
    \[
    \sigma
    \leq |p_{\bar{T}}-p_0|
    \leq 2\sqrt{n\alpha \bar{T}}
    <2\sqrt{n\alpha T}
    \leq\sigma,
    \]
    a contradiction.
    Hence $\bar{T}=T$.
    
    For the second statement, define 
    \[
    \hat{T}=\inf\left\{\,S\in[0,T] \mid q_t\in B_r(x_0)\cap M_t,\,\forall\,t\in[S,T]\,\right\}.
    \]
    For $t\in[\hat{T},T]$, by Lemma~\ref{lem:distance distortion},
    \[
    |q_t-q_T|\leq 2\sqrt{n\alpha T}-2\sqrt{n\alpha t},
    \]
    and thus 
    \[
    B_{\rho-(2\sqrt{n\alpha T}-2\sqrt{n\alpha t})}(q_t) 
    \subset B_\rho(q_T)
    \subset B_r(x_0)
    \]
    if $2\sqrt{n\alpha T}-2\sqrt{n\alpha t}\leq\rho$.
    Suppose that $\hat{T}>0$.
    By continuity and the definition of $\hat{T}$, $q_{\hat{T}}\in \partial B_r(x_0)$.
    Since $q_T\in B_r(x_0)\cap M_T$, we have
    \[
    \rho\leq |q_{\hat{T}}-q_T|
    \leq 2\sqrt{n\alpha T}-2\sqrt{n\alpha\hat{T}}
    <2\sqrt{n\alpha T}\leq\rho,
    \]
    a contradiction. 
    Hence $\hat{T}=0$.
\end{proof}

\iffalse
\begin{lemma}\label{lem:comparison on intrinsic distance}
    Let $M^n\subset\R^{n+m}$ be a submanifold with $\sup_{M}|A|\leq \lambda$.
    For any $p,q\in M$, if $|p-q|\leq \lambda^{-1}$, then 
    \[
    d_{g}(p,q)\leq 2|p-q|.
    \]
    In particular, for $p\in M$ and $\rho\leq \lambda^{-1}$,
    \[
    C_p\left(B_{\rho}(p)\cap M\right)
    \subset B_{2\rho}^M(p).
    \]
\end{lemma}
\begin{proof}
    Let $\gamma:[0,L]\to M\subset\R^{n+m}$ be the minimizing geodesic with $\gamma(0)=q$ and $\gamma(L)=p$.
    Let $\gamma_\lambda:[0,L]\to\R^{n+m}$ be a unit-speed geodesic with constant curvature $\lambda$.
    Since $|\gamma''|=|A(\gamma',\gamma')|\leq\lambda$, by Schur lemma, 
    \[
    d:=|\gamma(L)-\gamma(0)|\geq |\gamma_\lambda(L)-\gamma_\lambda(L)|=\frac2\lambda\sin(\frac{\lambda L}2).
    \]
    Thus, 
    \[
    L\leq \frac2\lambda\arcsin(\frac\lambda2 d),\quad\forall\,d\in[0,\frac2\lambda).
    \]
    When $\frac\lambda2 d\leq a\in[0,1)$, by convexity, 
    \[
    \frac2\lambda\arcsin(\frac\lambda2 d)
    \leq \frac2\lambda\cdot \frac{\arcsin{a}}a\cdot\frac\lambda2 d
    \leq \frac{\arcsin{a}}a\cdot d.
    \]
    In particular, for $a=\frac12$, if $d\leq\lambda^{-1}$, $L\leq \frac\pi3 d< 2d$.
\end{proof}
\fi

\subsection{Graphical representation}$\\$

If the second fundamental form is controlled, then a piece of the submanifold can be represented as a graph.
For later application to the local mean curvature flow, we modify Lemma~7.1 of \cite{MR2379801} and Lemma~3.3 of \cite{MR4022071} and provide an immersed version here.

\begin{lemma}\label{lem:small curvature graph lemma}
    There exist constants $\varepsilon=\varepsilon(n,m)$ and $C=C(n,m)$ such that the following holds.
    Let $r>0$ and let $F:M^n\to\R^{n+m}$ be a proper immersion in $B_{6r}(x_0)$ for some $x_0=F(p)$ with 
    \[
    \sup_{F^{-1}(B_{6r}(x_0))}|A|\leq\kappa,
    \qquad \kappa r\leq \varepsilon.
    \]
    Choose orthonormal coordinates $(x',x'')\in\R^n\times\R^m$ centered at $x_0$ such that 
    \[
    d F_p(T_p M)=\R^n\times\{0\}.
    \]
    Let
    \[
    \pi(x',x'')=x',\qquad \mathcal{C}_{4r}:=\{(x',x'')\mid|x'|<4r\},
    \]
    and define
    \[
    U=C_p\left(F^{-1}(B_{6r}(x_0)\cap \mathcal{C}_{4r})\right).
    \]
    Then
    \begin{enumerate}[label=(\alph*)]
    \item 
    \[
    \pi\circ F|_U:U\to B_{4r}^n
    \]
    is a diffeomorphism.
    \item
    $F|_U$ is represented by
    \[
    F\circ(\pi\circ F|_U)^{-1}(z)=(z,u(z)),
    \]
    where $u: B_{4r}^n\to\R^m$ is a smooth map satisfying
    \[
    u(0)=0,\qquad |Du(z)|\leq C\kappa|z|.
    \]
    In particular, $F|_{U}$ is an embedding.
    \item 
    \[
    C_p\left(F^{-1}(B_{3r}(x_0))\right)=U\cap F^{-1}(B_{3r}(x_0)).
    \]
    \item 
    For 
    \[
    Q_s:=U\cap F^{-1}(B_s(x_0)),\quad s\in(0,2r],
    \]
    we have
    \[
    (1-C\kappa^2s^2)\omega_n s^n\leq \mu_g(Q_s)\leq (1+C\kappa^2s^2)\omega_n s^n,\quad g=F^*g_e.
    \]
    \end{enumerate}
\end{lemma}
\begin{proof}
    By translation and scaling, it suffices to prove the case $x_0=0$, $r=1$.
    Set 
    \[
    h=\pi\circ F.
    \]
    Since $dh_p$ is invertible, the inverse function theorem provides a local inverse $\Psi$ of $h$ such that 
    \[
    \Psi(0)=p,\qquad F(\Psi(z))=(z,u(z)),
    \]
    where $u(0)=0$, $Du(0)=0$.
    
    First suppose that this parameterization is defined on $B_a^n$ with $a\leq4$ and that $|Du|\leq1$.
    The graphical formulas give   
    \begin{equation}\label{ineq:small curvature graph lemma-1}
    |D^2 u|\leq C_0|A|\leq C_0\kappa,
    \end{equation}
    where $C_0=C_0(n,m)$.
    Since $u(0)=0$ and $|Du|\leq1$, 
    \[
    |u(z)|\leq |z|, \qquad |(z,u(z))|\leq \sqrt2|z|\leq 4\sqrt2<6.
    \]
    Integrating \eqref{ineq:small curvature graph lemma-1} gives
    \begin{equation}\label{ineq:small curvature graph lemma-2}
        |Du(z)|\leq C_0\kappa|z|,\qquad |u(z)|\leq \frac{C_0}{2}\kappa|z|^2.
    \end{equation}
    Choosing $\epsilon$ sufficiently small yields the improved estimate 
    \begin{equation}\label{ineq:small curvature graph lemma-3}
        |Du(z)|\leq \tfrac14.
    \end{equation}

    Any two such extensions agree on their common domain by uniqueness of the local inverse and connectedness.
    Taking the union of such extensions gives
    \[
    \Psi:B_{a_*}^n\to M,
    \]
    where $a_*\leq4$ is the maximal radius of extension.
    We claim that $a_*=4$.
    Suppose otherwise.
    By \eqref{ineq:small curvature graph lemma-3}, 
    \[
    d_g(\Psi(z),\Psi(w))\leq \sqrt{1+\tfrac{1}{16}}|z-w|,\quad \forall\,z,w\in B_{a_*}^n.
    \]
    Since the image of $\Psi$ is contained in $F^{-1}(\overline{B}_{4\sqrt2})\Subset M^\circ$, this estimate implies that $\Psi$ extends continuously to $\overline{B}_{a_*}^n$.
    For any $q=\Psi(z)$ and any $v=d\Psi_z(a)\in T_q M$, by $F\circ\Psi(z)=(z,u(z))$ and \eqref{ineq:small curvature graph lemma-3} we have
    \[
    |v|_g^2=|d F_q(v)|^2=|a|^2+|Du(z)a|^2\leq (1+\tfrac{1}{16})|a|^2,
    \]
    which gives a uniform positive lower bound for the smallest singular value of $dh$,
    \begin{equation}\label{ineq:small curvature graph lemma-4}
    \sigma_{\min}(d h_q)\geq\frac{4}{\sqrt{17}}>0.
    \end{equation}
    By smoothness, \eqref{ineq:small curvature graph lemma-4} persists at the boundary limit points.
    The inverse function theorem therefore extends $\Psi$ across each boundary point.
    By compactness of the boundary and uniqueness of the local inverses, patching these extensions together gives an extension to a larger ball, still satisfying $|Du|<1$, contradicting maximality.
    Hence $a_*=4$.
    
    Now we obtain 
    \[
    \Psi: B_4^n\to M,\qquad h\circ\Psi=\mathrm{id},
    \]
    with the estimates \eqref{ineq:small curvature graph lemma-2}.
    The set $\tilde{U}:=\Psi(B_4^n)$ is open in $M$, and $h|_{\tilde{U}}:\tilde{U}\to B_4^n$
    is a diffeomorphism.
    Moreover, $\tilde{U}$ is connected and both open and closed in $V:=F^{-1}(B_{6}\cap\mathcal{C}_4)$.
    Since $p\in\tilde U$, we obtain $\tilde{U}=C_p(V)=U$.
    Thus,
    \[
    F|_U=G\circ h|_U,\qquad G(z)=(z,u(z)),
    \]
    and $F|_U$ is an embedding.

    Set $S:=U\cap F^{-1}(B_3)$, which is open and relatively closed in $F^{-1}(B_3)$.
    Since $u(0)=0$ and $|Du|\leq\frac14$, for any $z\neq0$ and $t\in(0,1]$, 
    \[
    \begin{aligned}
    \frac{d}{dt}|G(tz)|^2
    =2t|z|^2+2\langle u(tz),Du(tz)z\rangle
    \geq 2t|z|^2(1-\tfrac{1}{16})>0.
    \end{aligned}
    \]
    Thus, for each $\Psi(z)\in S$, the path $\{\Psi(tz) \mid t\in[0,1]\}$ lies in $S$.
    Hence $S$ is connected and 
    \[
    S=C_p(F^{-1}(B_3)).
    \]

    For $s\in(0,2]$, set
    \[
    V_s:=\{\,z\in B_4^n \mid |G(z)|^2<s^2\,\}.
    \]
    Then $Q_s=\Psi(V_s)$.
    By \eqref{ineq:small curvature graph lemma-2},
    \begin{equation}\label{ineq:small curvature graph lemma-5}
    B_{\frac{s}{\sqrt{1+C_1^2\kappa^2 s^2}}}^n\subset V_s\subset B_s^n,
    \end{equation}
    where $C_1=C_1(n,m)$.
    The graphical volume density on $V_s$ satisfies
    \begin{equation}\label{ineq:small curvature graph lemma-6}
    1\leq \sqrt{\det(I_n+(Du)^{\top}Du)}\leq 1+C_2\kappa^2 s^2,
    \end{equation}
    where $C_2=C_2(n,m)$.
    Integrating \eqref{ineq:small curvature graph lemma-6} over $V_s$ and using \eqref{ineq:small curvature graph lemma-5}, we obtain 
    \[
    \begin{aligned}
    \mu_g(Q_s)
    &\leq (1+C_2\kappa^2s^2)\omega_ns^n,\\
    \mu_g(Q_s)
    &\geq \omega_ns^n(1+C_1^2\kappa^2s^2)^{-n/2}
    \geq (1-C_3\kappa^2s^2)\omega_ns^n,
    \end{aligned}
    \]
    where $C_3=C_3(n,m)$.
\end{proof}

\subsection{Integral varifolds and curvature varifolds}$\\$

We begin by recalling the notion of integral varifolds.

\begin{definition}
Let $(N^{n+m},g)$ be a smooth Riemannian manifold, and let
$G_n(N)$ denote the Grassmann bundle of unoriented $n$-planes
in $TN$.
An $n$-varifold in $N$ is a Radon measure on $G_n(N)$.
It is called integral if it is induced by a pair $(M,\theta)$,
where $M\subset N$ is a countably $n$-rectifiable set and
$\theta:M\to\mathbb Z_{>0}$ is measurable, with
\[
\int_{M\cap K}\theta\,d\mathcal H^n<\infty
\qquad\text{for every compact set }K\subset N,
\]
through the identity
\[
\int_{G_n(N)}\varphi(x,S)\,dV(x,S)
=
\int_M\varphi(x,T_xM)\theta(x)\,d\mathcal H^n(x),
\qquad \varphi\in C_c(G_n(N)).
\]
Here $T_xM$ denotes the approximate tangent plane, defined
for $\mathcal H^n$-almost every $x\in M$.
We associate to $V$ its weight measure
\[
\mu_V(U)=\|V\|(U):=\int_{M\cap U}\theta\,d\mathcal H^n
\qquad\text{for every Borel set }U\subset N.
\]
\end{definition}

Next, we recall the definitions of generalized mean curvature and stationarity.

\begin{definition}
Let $V$ be an $n$-varifold in $N$, with weight measure $\mu_V$, and let $X\in C_c^1(N;TN)$.
Let $\Phi_t$ be the family of diffeomorphisms generated by $X$.
The first variation of $V$ along $X$ is defined by
\[
\delta V(X)
:=
\left.\frac{d}{dt}\right|_{t=0}
\|(\Phi_t)_\#V\|(W)
=
\int_{G_n(N)}\operatorname{div}_S X(x)\,dV(x,S),
\]
where $W\Subset N$ is any open set containing $\spt X$.
This definition is independent of the choice of $W$.
We say that $V$ has locally bounded first variation if $\delta V$ is represented by a vector-valued Radon measure, whose total variation is denoted by $\|\delta V\|$.

If, in addition, $\|\delta V\|\ll\mu_V$, the Radon--Nikodym theorem yields a vector field
$\mathbf H_V\in L_{\mathrm{loc}}^1(\mu_V;TN)$, which is called the generalized mean curvature vector of $V$, such that
\[
\delta V(X)
=
-\int_N \mathbf H_V\cdot X\,d\mu_V
\qquad\text{for all }X\in C_c^1(N;TN).
\]
Throughout this paper, saying that $V$ has generalized mean curvature includes the requirement $\|\delta V\|\ll\mu_V$.

If $\delta V(X)=0$ for every $X\in C_c^1(N;TN)$, then $V$ is called stationary in $N$.
\end{definition}

We will make use of the following fundamental compactness result due to Allard \cite{MR307015} for integral varifolds with locally uniformly bounded mass and first variation.

\begin{lemma}[Allard's compactness theorem]\label{lem:Allard compactness}
    Let $U\subset N$ be open, and let $\{V_j\}$ be a sequence of integral $n$-varifolds in $U$ such that
    \[\
    \sup_j\{\|V_j\|(W)+\|\delta V_j\|(W)\}<\infty,\quad \forall\,W\Subset U \text{ open}.
    \]
    Then there exists an integral $n$-varifold $V$ such that, after passing to a subsequence, 
    \[
    V_j\rightharpoonup V\quad \text{as Radon measures on }G_n(U).
    \]
    Moreover, 
    \[
    \|\delta V\|(W)\leq \liminf_{j\to\infty}\|\delta V_j\|(W),\quad\forall\,W\Subset U\text{ open}.
    \]
\end{lemma}

Allard also proved a regularity theorem that provides a local graphical description whenever the mean curvature is small in a suitable $L^p$ sense and the mass ratio is close to one.

\iffalse
\begin{lemma}[Allard's regularity theorem]\label{lem:Allard regularity}
For any $p>n$ there exist constants $\gamma=\gamma(n,m,p)\in(0,1)$, $\lambda_0=\lambda_0(n,m,p)\in (0,\frac{1}{16}]$ and $C=C(n,m,p)>0$ such that the following holds.
Fix $\lambda\in(0,\lambda_0]$.
Let $V$ be an integral $n$-varifold in $\R^{n+m}$.	
If
	\begin{equation*}
		\omega_n^{-1}\rho^{-n}\mu_V(B_\rho)\leq 1+\lambda,
        \quad \bigg(\rho^{p-n}\int_{B_{\rho}}|\mathbf H|^p\, d\mu_V\bigg)^{1/p}\leq \lambda,
	\end{equation*}
then there is an orthogonal transformation $Q$ of $\mathbb{R}^{n+m}$ and $u=(u^1,\dots, u^m)\in C^{1,1-n/p}$ such that 
    \[
    Du(0)=0,\quad \spt V\cap B_{\gamma\rho}=Q(\graph u)\cap B_{\gamma\rho},
    \]
and 
	\begin{equation*}
		\begin{split}
			\rho^{-1}\sup |u|+\sup |Du|+\rho^{1-\frac{n}{p}}\sup\limits_{x,y\in B_{\gamma\rho},x\neq y}\frac{|Du(x)-Du(y)|}{|x-y|^{1-n/p}}\leq C\lambda^{\frac{1}{2n+2}}.
		\end{split}
	\end{equation*}
\end{lemma}
\fi

\begin{lemma}[Allard's regularity theorem]\label{lem:Allard regularity}
Let $p>n$ and set $\alpha=1-n/p$.
There exist constants $\lambda_0\in(0,1/16], \gamma\in(0,1)$ and $C>0$, depending only on $n$, $m$, and $p$, with the following property.
Let $V$ be an integral $n$-varifold in $B_\rho\subset\R^{n+m}$ with $0\in\spt\|V\|$ and generalized mean curvature $\mathbf H_V\in L^p(B_\rho,\|V\|)$.
If, for some $\lambda\in(0,\lambda_0]$,
\[
\frac{\|V\|(B_\rho)}{\omega_n\rho^n}\leq 1+\lambda,
\qquad
\rho^{\alpha}
\|\mathbf H_V\|_{L^p(B_\rho,\,\|V\|)}\leq\lambda,
\]
then there exist an $n$-dimensional linear subspace $P\subset\R^{n+m}$ and a map
$u\in C^{1,\alpha}(B_{\gamma\rho}^P;P^\perp)$
such that
\[
u(0)=0,\qquad
Du(0)=0,\qquad
\spt \|V\|\cap B_{\gamma\rho}=\graph_P u\cap B_{\gamma\rho}.
\]
Moreover,
\[
\rho^{-1}\|u\|_{L^\infty}
+\|Du\|_{L^\infty}
+\rho^\alpha[Du]_{C^{0,\alpha}}
\leq C\lambda^{\frac{1}{2n+2}},
\]
where all norms and seminorms are taken on $B_{\gamma\rho}^P$.
\end{lemma}

%Since the second fundamental form is not preserved under varifold convergence, 
To pass curvature-energy bounds to varifold limits, we recall the notion of curvature varifolds. 
Curvature varifolds were introduced by Hutchinson \cite{MR825628} as a weak framework for studying possibly singular submanifolds with an $L^p$-integrable second fundamental form. 

\begin{definition}
Let $V$ be an integral $n$-varifold in an open set $\Omega\subset\mathbb{R}^{n+m}$. 
We say that $V$ is a curvature varifold if there exist functions
\[
\mathcal{A}_{ijk}\in L^1_{\mathrm{loc}}(V),\quad i,j,k\in\{1,\dots,n+m\},
\]
such that, for every $\varphi\in C_c^1(\Omega\times\R^{(n+m)^2})$,
\[
\int_{G_n(\Omega)}
\sum_{j,k}
\Bigl(
P_{ij}\frac{\partial\varphi}{\partial x_j}
+
\mathcal A_{ijk}\frac{\partial\varphi}{\partial P_{jk}}
+
\mathcal A_{jij}\varphi
\Bigr)\,dV(x,P)=0,\quad i\in\{1,\dots,n+m\}.
\]
Here, each plane $P\in G(n,n+m)$ is identified with its orthogonal projection matrix $(P_{ij})$.
\end{definition}
The tensor $\mathcal A_{ijk}$ represents the weak tangential derivative of the tangent-plane projection $P$ and determines a generalized second fundamental form
\[
(B_V)_{ij}^{k}
= \sum_{l} P_{jl}\mathcal A_{ik l}.
\]
When $V$ is induced by a smooth submanifold, $B_V$ agrees with the classical second fundamental form $A$, and $|\mathcal{A}|^2=2|A|^2$.
Moreover, we have the following compactness and lower-semicontinuity properties for curvature varifolds. 

\begin{lemma}[\cite{MR1412686}]\label{lem:compactness and lower-semicontinuity of curvature varifolds}
Let $\{V_l\}$ be a sequence of curvature varifolds in
an open set $\Omega\subset\mathbb R^{n+m}$, with associated
tensors $\mathcal A^{(l)}$.
Suppose that, for some $p>1$ and every open set $W\Subset\Omega$,
\[
\sup_l\left\{
\|V_l\|(W)
+\int_{G_n(W)}|\mathcal A^{(l)}|^p\,dV_l
\right\}<\infty.
\]
Then, after passing to a subsequence, there exists a curvature
varifold $V$ in $\Omega$, with
$\mathcal A\in L_{\mathrm{loc}}^p(V)$, such that
\[
V_l\rightharpoonup V,\qquad
\mathcal A_{ijk}^{(l)}V_l
\rightharpoonup\mathcal A_{ijk}V\quad
\text{as Radon measures on }G_n(\Omega).
\]
Moreover, for every nonnegative $\zeta\in C_c(\Omega)$,
\[
\int_{G_n(\Omega)}
\zeta(x)|\mathcal A|^p\,dV
\leq
\liminf_{l\to\infty}
\int_{G_n(\Omega)}
\zeta(x)|\mathcal A^{(l)}|^p\,dV_l.
\]
\end{lemma}

\subsection{Interior estimates for self-shrinkers}
\begin{proposition}\label{prop:interior estimate for self-shrinker}
Let $\Sigma^n\subset \R^{n+1}$ be a smooth self-shrinker satisfying
\begin{equation}\label{eq:shrinker}
    \mathbf H+\tfrac12\mathbf x^{\perp}=0.
\end{equation}
Fix $p\in \R^{n+1}$, $0<\rho\le 1$, and an $n$-dimensional linear hyperplane
$P\subset \R^{n+1}$ with unit normal $\mathbf n_0$. 
Assume that a neighborhood of $\Sigma$ near $p$ is represented as the graph
\begin{equation*}%\label{eq:graph}
    F(z)=p+z+u(z)\mathbf n_0,
    \quad z\in B_{\rho}^P.
\end{equation*}
Suppose that
\begin{equation}\label{eq:shrinker-graph-bound}
    \rho^{-1}\|u\|_{L^\infty(B^P_\rho)}
    +\|Du\|_{L^\infty(B^P_\rho)}\le L,
    \qquad
    \rho |p|\le R.
\end{equation}
Then, for every $0<\theta<1$ and every integer $k\ge 2$, there exists a constant
\(
    C_k=C_k(n,k,\theta,L,R)<\infty
\)
such that
\begin{equation}\label{eq:ukestimate}
    \sup_{B^P_{\theta\rho}} |D^k u|
    \le C_k\,\rho^{1-k}.
\end{equation}
In particular, 
\begin{equation}\label{eq:Aestimate}
    \sup_{F(B^P_{\theta\rho})}|\nabla^k A|
    \le C_k\,\rho^{-k-1},\quad
    \forall\,k\geq0.
\end{equation}
\end{proposition}

\begin{proof}
Write
\[
    p=p^T+p^\perp\mathbf n_0,
    \qquad p^T\in P,
    \qquad p^\perp=\langle p,\mathbf n_0\rangle.
\]
After identifying $P$ with $\R^n$, the graph is
\[
    F(z)=(p^T+z,\,p^\perp+u(z)),
\]
and its upward unit normal is
\[
    \mathbf n=\frac{(-Du,1)}{\sqrt{1+|Du|^2}}.
\]
Hence
\[
    \langle F,\mathbf n\rangle
    =\frac{p^\perp+u-(p^T+z)\cdot Du}
    {\sqrt{1+|Du|^2}}.
\]
Moreover,
\[
    \langle \mathbf H,\mathbf n\rangle
    =\Div\left(\frac{Du}{\sqrt{1+|Du|^2}}\right)
    =\frac{1}{\sqrt{1+|Du|^2}}
      a^{ij}(Du)u_{ij},
\]
where
\begin{equation*}%\label{eq:aij}
    a^{ij}(q)=\delta_{ij}-\frac{q_iq_j}{1+|q|^2}.
\end{equation*}
Therefore \eqref{eq:shrinker} is equivalent to
\begin{equation}\label{eq:graphPDE}
    a^{ij}(Du)u_{ij}
    =\tfrac12\left((p^T+z)\cdot Du-p^\perp-u\right).
\end{equation}
Since $|Du|\le L$, the coefficient matrix is uniformly elliptic:
\begin{equation*}%\label{eq:ellipticity}
    \frac{1}{1+L^2}|\xi|^2
    \le a^{ij}(Du)\xi_i\xi_j
    \le |\xi|^2,
    \quad\forall\,\xi\in\R^n.
\end{equation*}

Now rescale to the unit ball by setting
\[
    z=\rho y,
    \qquad v(y)=\rho^{-1}u(\rho y),
    \qquad y\in B^P_1.
\]
Then
\[
    Dv(y)=Du(\rho y),
    \qquad D_y^k v=\rho^{k-1}D_z^k u,
\]
and \eqref{eq:graphPDE} becomes
\begin{equation}\label{eq:scaledPDE}
    a^{ij}(Dv)v_{ij}
    =\tfrac12\Big(b\cdot Dv-c\Big)
      +\tfrac12\rho^2\Big(y\cdot Dv-v\Big),
\end{equation}
where $b:=\rho p^T$ and $c:=\rho p^\perp$.
By \eqref{eq:shrinker-graph-bound},
\[
    |b|,|c|\leq R,\qquad
    \|v\|_{L^\infty(B_1)}+\|Dv\|_{L^\infty(B_1)}\leq L.
\]
Equation \eqref{eq:scaledPDE} is uniformly elliptic, with ellipticity controlled by $L$ and lower-order coefficients bounded in terms of $L$ and $R$.
Standard interior quasilinear elliptic estimates first give, for every
$0<\theta_1<1$,
\[
    \|v\|_{C^{1,\alpha}(B_{\theta_1})}
    \le C(n,\theta_1,L,R)
\]
for some $\alpha\in(0,1)$. Consequently $a^{ij}(Dv)$ is uniformly
$C^{\alpha}$ on smaller balls, and the Schauder estimate applied to
\eqref{eq:scaledPDE} yields
\[
    \|v\|_{C^{2,\alpha}(B_{\theta_2})}
    \le C(n,\theta_2,L,R)
\]
for each $\theta_2<\theta_1$. By bootstrapping, for every integer $k\ge2$ and every
$0<\theta<1$,
\begin{equation*}%\label{eq:vkestimate}
    \|D^k v\|_{L^\infty(B_\theta)}
    \le C_k(n,k,\theta,L,R).
\end{equation*}
Scaling back with $D_y^k v=\rho^{k-1}D_z^k u$ proves
\eqref{eq:ukestimate}.

Finally, on a graph,
\[
    g_{ij}=\delta_{ij}+u_i u_j,
    \qquad
    h_{ij}=\frac{u_{ij}}{\sqrt{1+|Du|^2}}.
\]
In local coordinates, the components of $\nabla^k A$ are linear combinations of products of derivatives of $u$ of orders $2,\ldots,k+2$, with coefficients depending smoothly on $Du$.
By \eqref{eq:shrinker-graph-bound} and \eqref{eq:ukestimate}, each term is bounded by $C_k\rho^{-k-1}$.
Since the induced metric is uniformly equivalent to the Euclidean metric, \eqref{eq:Aestimate} follows.
\end{proof}

\section{A priori estimates for mean curvature flow}\label{sec3}

\subsection{Interior estimates under curvature decay}$\\$

Following the argument in Theorem~3.1 of \cite{MR2520796}, we show that the preliminary $t^{-1}$ curvature bound suffices to establish the compactness of mean curvature flows evolving from a common initial submanifold.

\begin{proposition}\label{prop:interior estimate by c/t}
    Fix $x_0\in\R^{n+m}$, $\rho>0$ and $\alpha>0$.
    Let $\{M_t^n\}_{t\in[0,T]}$ be a smooth mean curvature flow in $\R^{n+m}$ satisfying
    \begin{enumerate}[label=(\roman*)]
        \item $B_{\rho}(x_0)\cap M_t\Subset M_t,\quad\forall\,t\in[0,T]$;
        \item $\sup_{B_{\rho}(x_0)\cap M_0}|A|^2\leq \rho^{-2}$;
        \item $\sup_{B_{\rho}(x_0)\cap M_t}|A|^2\leq \alpha t^{-1}, \quad\forall\,t\in(0,T].$
    \end{enumerate}
    Then there exists a constant $C=C(n,m,\alpha)>0$ such that
    \[
    |A|^2(x,t)\leq e^{C\alpha}(\rho-|x-x_0|)^{-2},
    \quad\forall\,t\in[0,T],\quad\forall\,x\in B_\rho(x_0)\cap M_t.
    \]
\end{proposition}
\begin{proof}
    It suffices to prove the case $\rho=1$.
    For $t\geq1$, assumption (iii) gives $|A|^2\leq\alpha$.
    Let $\gamma\geq1$ such that $e^{\gamma\alpha}\geq\alpha$.
    Suppose that there exist $\theta\in(0,1)$, $t_1\in(0,T\wedge1]$ and $x_1\in\partial B_{1-\theta}(x_0)\cap {M}_{t_1}$ such that $|A|^2(x_1,t_1)\geq e^{\gamma\alpha}\theta^{-2}$.
    Fix some $K>0$ to be determined later. 
    Starting from $(x_1,t_1)$, choose successive points $\left\{(x_k,t_k)\right\}$ satisfying the following conditions, as long as such a choice is possible:
    \begin{enumerate}[label=(\alph*)]
        \item $0\leq t_k\leq t_{k-1}$;
        \item $|x_{k}-x_0|\leq |x_{k-1}-x_0|+K\alpha|A|^{-1}(x_{k-1},t_{k-1})$;
        \item $|A|^2(x_{k},t_{k})\geq 4|A|^2(x_{k-1},t_{k-1})$.
    \end{enumerate}
    Since
    \[
    |A|^2(x_k,t_k)\geq 4^{k-1}|A|^2(x_1,t_1)\geq 4^{k-1}e^{\gamma\alpha}\theta^{-2},
    \]
    we have
    \[
    |x_k-x_0|
    \leq |x_1-x_0|+K\alpha\sum_{i=1}^{k-1}2^{-(i-1)}e^{-\frac{\gamma\alpha}{2}}\theta
    \leq 1-\theta+2K\alpha e^{-\frac{\gamma\alpha}{2}}\theta
    \leq 1-\tfrac\theta2,
    \]
    if $K\alpha e^{-\frac{\gamma\alpha}{2}}\leq\frac14$ for $\gamma=\gamma(K,\alpha)$ sufficiently large.
    Thus, this sequence is finite, and there exist $\bar{t}\in(0,t_1]$ and $\bar{x}\in B_{1-\theta/2}(x_0)\cap {M}_{\bar{t}}$ such that 
    \begin{equation}\label{ineq:interior estimate by c/t-1}
    \lambda:=|A|^2(\bar{x},\bar{t})\geq e^{\gamma\alpha}\theta^{-2},
    \end{equation}
    \begin{equation}\label{ineq:interior estimate by c/t-2}
    \sup_{t\in[0,\bar{t}]}\sup_{B_{|\bar{x}-x_0|+K\alpha\lambda^{-1/2}}(x_0)\cap{M}_t} |A|^2(\cdot,t)\leq 4\lambda,
    \end{equation}
    where $|\bar{x}-x_0|+K\alpha\lambda^{-1/2}\leq 1-\frac\theta2$.
    Let $\varphi:\R\to[0,1]$ be a smooth nonincreasing cutoff function such that
    \[
    \varphi|_{(-\infty,\frac12]}\equiv1,\quad 
    \varphi|_{[1,\infty)}\equiv0,\quad
    |\varphi''|\varphi^{-\frac12}+|\varphi'|^2\varphi^{-\frac32}\leq c.
    \]
    Define 
    \[
    \eta=\frac{|\mathbf x_t-x_0|-|\bar{x}-x_0|}{K\alpha\lambda^{-1/2}},\qquad
    \phi=(\varphi\circ\eta)|A|^2.
    \]
    Then $\varphi'(\eta)$ and $\varphi''(\eta)$ are supported in $A_{|\bar{x}-x_0|+\frac12K\alpha\lambda^{-1/2},|\bar{x}-x_0|+K\alpha\lambda^{-1/2}}(x_0)$.
    Since 
    \[
    \square|\mathbf x_t|=-\frac{n}{|\mathbf x_t|}+\frac{|\mathbf x_t^\top|^2}{|\mathbf x_t|^3},\quad
    |\nabla|\mathbf x_t||^2=\frac{|\mathbf x_t^\top|^2}{|\mathbf x_t|^2}, 
    \]
    on this annulus we have 
    \begin{equation}\label{ineq:interior estimate by c/t-3}
    -\square\eta\leq \frac{n(K\alpha)^{-1}\lambda^{\frac12}}{|\mathbf x_t-x_0|}\leq 2n(K\alpha)^{-2}\lambda,\quad 
    |\nabla\eta|^2\leq (K\alpha)^{-2}\lambda.
    \end{equation}
    By H\"{o}lder's inequality and the evolution of $|A|^2$, 
    \begin{equation}\label{ineq:interior estimate by c/t-4}
    \begin{aligned}
    &\square\phi\\
    =&\varphi'\square\eta\cdot|A|^2-\varphi''|\nabla\eta|^2\cdot|A|^2+\varphi\cdot\square|A|^2-2\varphi'\nabla\eta\cdot\nabla|A|^2\\
    \leq& -\sqrt{c}\varphi^{\frac34}\square\eta\cdot|A|^2
    +c\varphi^{\frac12}|\nabla\eta|^2|A|^2
    +\varphi(-2|\nabla A|^2+C_1|A|^4)\\
    &+2\varphi|\nabla A|^2+2\varphi'^2\varphi^{-1}|\nabla\eta|^2|A|^2\\
    \leq& -\sqrt{c}\varphi^{\frac34}\square\eta\cdot|A|^2
    +3c\varphi^{\frac12}|\nabla\eta|^2|A|^2
    +C_1|A|^2\phi,
    \end{aligned}
    \end{equation}
    where $C_1=C_1(n,m)$.
    Substituting \eqref{ineq:interior estimate by c/t-2} and \eqref{ineq:interior estimate by c/t-3} into \eqref{ineq:interior estimate by c/t-4} yields
    \[
    \begin{aligned}
    \square\phi
    &\leq 8n\sqrt{c}(K\alpha)^{-2}\lambda^{2}
    +12c(K\alpha)^{-2}\lambda^2
    +4C_1\lambda\phi\\
    &\leq C(n)(K\alpha)^{-2}\lambda^2+4C_1\lambda\phi.
    \end{aligned}
    \]
    Then $\psi(t):=\sup_{{M}_t}\phi(\cdot,t)$ is Lipschitz and satisfies
    \[
    \psi(0)\leq1,\quad \frac{d}{dt}\psi\leq 4C_1\lambda\psi+C(n)(K\alpha)^{-2}\lambda^2
    \]
    almost everywhere.
    By the definition of $\lambda$ and the fact $\lambda\bar{t}\leq\alpha$, we have
    \[
    \begin{aligned}
    \lambda
    &\leq \psi(\bar{t})
    \leq e^{4C_1\lambda\bar{t}}+C(n)(K\alpha)^{-2}\lambda^2\frac{e^{4C_1\lambda\bar{t}}-1}{4C_1\lambda}\\
    &\leq e^{4C_1\alpha}+C(n)\frac{e^{4C_1\alpha}-1}{4C_1(K\alpha)^2}\lambda,
    \end{aligned}
    \]
    which implies 
    \[
    \lambda\leq 2e^{4C_1\alpha}
    \]
    by choosing $K=K(n,m,\alpha)$ sufficiently large.
    Since $\lambda> e^{\gamma\alpha}$ by \eqref{ineq:interior estimate by c/t-1}, choosing $\gamma=\gamma(K,\alpha)$ sufficiently large causes a contradiction.
\end{proof}

\begin{corollary}[Curvature control up to initial time]\label{cor:interior estimate by c/t}
    Fix $x_0\in\R^{n+m}$, $\rho>0$ and $\alpha>0$.
    Let $\{M_t\}_{t\in[0,T]}$ be a smooth mean curvature flow in $\R^{n+m}$ satisfying
    \begin{enumerate}[label=(\roman*)]
        \item $B_{\rho}(x_0)\cap {M}_0\Subset{M}_0$;
        \item $\sup_{B_\rho(x_0)\cap M_0}|A|^2\leq \rho^{-2}$;
        \item $\sup_{B_{\rho}(x_0)\cap {M}_t}|A|^2\leq \alpha t^{-1}, \quad\forall\,t\in(0,T].$
    \end{enumerate}
    Then there exists a constant $C=C(n,m,\alpha)>0$ such that
    \[
    |A|^2(x,t)\leq 16e^{C\alpha}(\rho-|x-x_0|)^{-2},
    \quad\forall\,t\in[0,T],\quad\forall\,x\in B_{\rho}(x_0)\cap {M}_t.
    \]
    Moreover, if assumption (ii) is removed, then there exists a constant $\tilde{\rho}=\tilde{\rho}(\rho,x_0,{M}_0)\in(0,\rho]$ such that 
    \[
    \sup_{B_{\tilde{\rho}}(x_0)\cap{M}_t}|A|^2(\cdot,t)\leq 16 e^{C\alpha}\tilde{\rho}^{-2},
    \quad\forall\,t\in[0,T].
    \]
\end{corollary}
\begin{proof}
    Fix $\theta\in(0,1)$ and $t_0\leq \frac{1}{64n\alpha}(1-\theta)^2\rho^2$ so that $2\sqrt{n\alpha t_0}\leq\frac14(1-\theta)\rho$.
    For any $p_{t_0}=\mathbf x_{t_0}(p)\in\partial B_{\theta\rho}(x_0)\cap{M}_{t_0}$, by Lemma~\ref{lem:shrinking ball}, 
    \[
    B_{(1-\theta)\rho-2\sqrt{n\alpha t_0}}(p_t)
    \subset B_{(1-\theta)\rho}(p_{t_0})
    \subset B_\rho(x_0),
    \quad\forall\,t\in[0,t_0].
    \]
    For any $t\leq t_0$ and any $q_t=\mathbf x_t(q)\in B_{(1-\theta)\rho-6\sqrt{n\alpha t_0}}(p_{t_0})\cap{M}_t$, 
    $|q_t-q_0|\leq 2\sqrt{n\alpha t}\leq 2\sqrt{n\alpha t_0}$, $|p_{t_0}-p_0|\leq 2\sqrt{n\alpha t_0}$, and thus 
    \[
    q_0\in B_{(1-\theta)\rho-2\sqrt{n\alpha t_0}}(p_0)
    \subset B_\rho(x_0).
    \]
    Hence, 
    \[
    B_{\frac14(1-\theta)\rho}(p_{t_0})\cap{M}_t
    \subset \Phi_{0,t}\left(B_\rho(x_0)\cap{M}_0\right)
    \Subset \Phi_{0,t}({M}_0)={M}_t,
    \quad\forall\,t\in[0,t_0].
    \]
    Applying Proposition~\ref{prop:interior estimate by c/t} yields 
    \[
    |A|^2(p_{t_0},t_0)\leq e^{C\alpha}\left(\tfrac14(1-\theta)\rho\right)^{-2},
    \]
    where $C=C(n,m,\alpha)$, which implies
    \[
    \sup_{\partial B_{\theta\rho}(x_0)\cap{M}_t}|A|^2(\cdot,t)\leq 16e^{C\alpha}(\rho-\theta\rho)^{-2},
    \quad\forall\,t\leq \frac{1}{64n\alpha}(1-\theta)^2\rho^2.
    \]
    On the other hand, for $t\geq\frac{1}{64n\alpha}(1-\theta)^2\rho^2$, by the assumption, 
    \[
    \sup_{\partial B_{\theta\rho}(x_0)\cap{M}_t}|A|^2(\cdot,t)
    \leq \alpha t^{-1}
    \leq 64n\alpha^2(1-\theta)^{-2}\rho^{-2}
    \leq e^{C\alpha}(\rho-\theta\rho)^{-2}.
    \]
    In summary, 
    \[
    \sup_{\partial B_{\theta\rho}(x_0)\cap{M}_t}|A|^2(\cdot,t)
    \leq 16e^{C\alpha}(\rho-\theta\rho)^{-2},
    \quad\forall\,\theta\in(0,1),
    \quad\forall\,t\in[0,T].
    \]
    For the final assertion, choose $\tilde\rho\in(0,\rho/2)$ sufficiently small that
    \[
    \sup_{B_{2\tilde\rho}(x_0)\cap M_0}|A|^2
    \leq (2\tilde\rho)^{-2}.
    \]
    Applying the first assertion with radius $2\tilde\rho$ gives the desired estimate on $B_{\tilde\rho}(x_0)$.
\end{proof}

\subsection{Weak monotonicity of total curvature}

\begin{lemma}\label{lem:integration by parts}
    Let $\{M_t^n\}_{t\in[0,T]}$ be a smooth mean curvature flow in $\R^{n+m}$.
    Fix $a\geq1$, $\epsilon>0$, a smooth compactly supported spacetime function $\phi=\phi(x,t)$ and the backward heat kernel
    \[
    \varrho=\varrho_{y,T_*}=(4\pi(T_*-t))^{-\frac n2}\exp\left(-\frac{|\mathbf x-y|^2}{4(T_*-t)}\right)
    \]
    for some $y\in\R^{n+m}$ and $T_*>0$.
    Then, for $t\in[0,T]\cap[0,T_*)$, we have
    \[
    \begin{aligned}
    \frac{d}{dt}\int_{M_t}\phi^2(|A|^2+\epsilon)^a\varrho \,d\mu_t
    \leq& 
    -\int_{M_t}\phi^2(|A|^2+\epsilon)^a\left|\mathbf H+\frac{(\mathbf x-y)^\perp}{2(T_*-t)}\right|^2\varrho \,d\mu_t\\
    &-C(a)\int_{M_t}|\nabla(\phi(|A|^2+\epsilon)^{a/2})|^2\varrho \,d\mu_t\\
    &+C(n,m,a)\int_{M_t} \phi^2(|A|^2+\epsilon)^{a+1}\varrho \,d\mu_t\\
    &+C(a)\int_{M_t} |\nabla\phi|^2(|A|^2+\epsilon)^a\varrho \,d\mu_t\\
    &+\int_{M_t} 2\phi\square\phi(|A|^2+\epsilon)^a\varrho \,d\mu_t.
    \end{aligned}
    \]
\end{lemma}
\begin{proof}
    Write $f=|A|^2+\epsilon$.
    By the evolution of the second fundamental form, 
    \begin{gather*}
    \square f\leq -2|\nabla A|^2+C(n,m)f^2,\\
    \square f^a\leq -C(a)|\nabla f^{a/2}|^2+C(n,m,a)f^{a+1}.
    \end{gather*}
    By Huisken's monotonicity formula, 
    \[
    \begin{aligned}
    \frac{d}{dt}\int_{M_t}\phi^2f^a\varrho \,d\mu_t
    &=\int_{M_t}\biggl[\square(\phi^2 f^a)-\phi^2f^a\Bigl\lvert\mathbf H+\frac{(\mathbf x-y)^\perp}{2(T_*-t)}\Bigr\rvert^2\biggr]\varrho \,d\mu_t.
    \end{aligned}
    \]
    By Young's inequality, 
    \[
    \begin{aligned}
    \square(\phi^2f^a)
    &=\phi^2\square f^a+f^a\square\phi^2-2\langle\nabla\phi^2,\nabla f^a\rangle\\
    &=\phi^2\square f^a+f^a(2\phi\square\phi-2|\nabla\phi|^2)-8\phi f^{a/2}\langle\nabla\phi,\nabla f^{a/2}\rangle\\
    &\leq -C(a)\phi^2|\nabla f^{a/2}|^2+C(n,m,a)\phi^2f^{a+1}+C(a)|\nabla\phi|^2f^a+2\phi\square\phi f^a.
    \end{aligned}
    \]
    Since 
    \[
    |\nabla(\phi f^{a/2})|^2\leq 2\phi^2|\nabla f^{a/2}|^2+2|\nabla\phi|^2 f^a,
    \]
    we get
    \[
    \square(\phi^2f^a)
    \leq -C(a)|\nabla (\phi f^{a/2})|^2
    +C(n,m,a)\phi^2f^{a+1}
    +C(a)|\nabla\phi|^2f^a+2\phi\square\phi f^a.
    \]
    Combining the estimates above yields the desired result.
\end{proof}

The following estimate bounds the energy error by a Gaussian tail.

\begin{lemma}\label{lem:weak local monotonicity 1}
There exist constants $\delta_1=\delta_1(n,m)$ and $T_1=T_1(n)$ such that the following holds.
Fix $x_0\in\mathbb R^{n+m}$, $\rho>0$, and $\delta\in(0,\delta_1]$. 
Let $\{M_t^n\}_{t\in[0,T]}$ be a smooth mean curvature flow in $\mathbb R^{n+m}$ satisfying
\begin{enumerate}[label=(\roman*)]
\item $B_\rho(x_0)\cap M_t\Subset M_t,\quad\forall\,t\in[0,T]$;
\item $\|A\|_{L^n(B_\rho(x_0)\cap M_t)}\leq\delta,\quad\forall\,t\in[0,T].$
\end{enumerate}
Then, for any $t_0\in(0,T\wedge T_1\rho^2]$ and any $y\in B_{\rho/2}(x_0)$,
\[
\|A\|_{L^n(B_{\sqrt{t_0}}(y)\cap M_t)}
\leq C(n,m)\left[\|A\|_{L^n(B_\rho(x_0)\cap M_0)}+\delta e^{-\frac{\rho^2}{512n t_0}}\right],\quad\forall\,t\in[0,t_0].
\]
\end{lemma}
    \begin{proof}
    It suffices to prove the case $\rho=1$.
    Let $\varphi:[0,\infty)\to[0,1]$ be a smooth nonincreasing cutoff function such that 
    \[
    \varphi\equiv1 \text{ on } [0,\tfrac12],\qquad
    \varphi\equiv0 \text{ on } [\tfrac{9}{16},\infty),\qquad
    |\varphi'|\varphi^{-\frac12}\leq c,\qquad
    -\varphi''\varphi^{-1}\leq c.
    \]
    Define 
    \[
    \eta:=|\mathbf x_t-x_0|^2+2nt
    \]
    with
    \[
    \square\eta=0,\qquad |\nabla\eta|^2\leq 4\eta. 
    \]
    Define the cutoff function 
    \begin{equation}\label{ineq:weak local monotonicity 1-0}
    \phi:=e^{-4ct}\varphi\circ\eta
    \end{equation}
    satisfying
    \begin{enumerate}[label=(\alph*)]
    \item $\spt\phi\subset B_{3/4}(x_0)\cap{M}_t$;
    \item $\phi\equiv e^{-4ct}$ in $B_{5/8}(x_0)\cap{M}_t$ for $t\leq \frac{7}{128n}$;
    \item $\spt|\nabla\phi|\subset A_{5/8,3/4}(x_0)\cap M_t$ for $t\leq \frac{7}{128n}$;
    \item $\square\phi=e^{-4ct}(-4c\varphi+\varphi'\square\eta-\varphi''|\nabla\eta|^2)\leq 4c\varphi(\eta-1)e^{-4ct}\leq0$;
    \item $|\nabla\phi|^2=|e^{-4ct}\varphi'\nabla\eta|^2\leq e^{-8ct}c^24\eta\varphi\leq 4c^2e^{-4ct}\phi$.
    \end{enumerate}
    Fix $t_0\in[0,T]$ and $y\in B_{1/2}(x_0)$.
    Set
    \[
    \tau=2t_0-t,\quad X=\mathbf x_t-y,
    \]
    and define 
    \[
    \varrho(x,t):=\left(4\pi\tau\right)^{-\frac n2}\exp\left(-\frac{|X|^2}{4\tau}\right).
    \]
    Consider 
    \[
    \mathcal{E}(t)=\int_{M_t}\phi^2|A|^n\varrho \,d\mu_t.
    \]
    Applying Lemma~\ref{lem:integration by parts} with $a=\frac n2$ and $\epsilon\downarrow0$, we have, for $t\in[0,t_0\wedge\frac{7}{128n}]$,
    \begin{equation}\label{ineq:weak local monotonicity 1-1}
    \begin{aligned}
    \mathcal{E}'
    \leq& 
    -\int_{M_t}\phi^2|A|^n\left|\mathbf H+\frac{X^\perp}{2\tau}\right|^2\varrho
    -C(n)\int_{M_t}\left|\nabla(\phi|A|^{n/2})\right|^2\varrho\\
    &+C(n,m)\int_{M_t} \phi^2|A|^{n+2}\varrho
    +C(n)\int_{M_t}|\nabla\phi|^2 |A|^n\varrho\\
    :=& -\mathrm{I}-C(n)\mathrm{II}+C(n,m)\mathrm{III}+C(n)\mathrm{IV}.
    \end{aligned}
    \end{equation}
    By Cauchy-Schwarz inequality,
    \[
    \begin{aligned}
    -\mathrm{II}
    %=-\int_{M_t}|\nabla(\phi|A|^{n/2})|^2\varrho
    &\leq-\tfrac12\int_{M_t}\left|\nabla(\phi|A|^{n/2}\varrho^{1/2})\right|^2
    +\int_{M_t}\phi^2|A|^n|\nabla\varrho^{1/2}|^2\\
    &=-\tfrac12\int_{M_t}\left|\nabla(\phi|A|^{n/2}\varrho^{1/2})\right|^2
    +\tfrac14\int_{M_t}\phi^2|A|^n\varrho^{-1}|\nabla\varrho|^2\\
    &:=-\tfrac12\mathrm{V}+\tfrac14\mathrm{VI}.
    \end{aligned}
    \]
    Note that
    \[
    \Div(\varrho^{-1}\nabla\varrho)
    =\Div\Bigl(-\frac{X^\top}{2\tau}\Bigr)
    =-\frac{1}{2\tau}(n+\langle X^\perp,\mathbf H\rangle).
    \]
    By the divergence theorem and Cauchy-Schwarz inequality,
    \[
    \begin{aligned}
    \mathrm{VI}
    =&\int_{M_t}\phi^2|A|^n\langle\nabla\varrho,\varrho^{-1}\nabla\varrho\rangle\\
    =&-\int_{M_t}\langle\nabla(\phi|A|^{n/2})^2,\nabla\varrho\rangle
    +\int_{M_t}\phi^2|A|^n\frac{1}{2\tau}(n+\langle X^\perp,\mathbf H\rangle)\varrho\\
    \leq& 2\int_{M_t}|\nabla(\phi|A|^{n/2})|^2\varrho
    +\tfrac12\int_{M_t}\phi^2|A|^n\varrho^{-1}|\nabla\varrho|^2+\frac{n}{2\tau}\mathcal{E}+\tfrac12\int_{M_t}\phi^2|A|^n\left|\mathbf H+\frac{X^\perp}{2\tau}\right|^2\varrho\\
    =& 2\mathrm{II}+\tfrac12\mathrm{VI}+\frac{n}{2\tau}\mathcal{E}+\tfrac12\mathrm{I}.
    \end{aligned}
    \]
    Hence, we have
    \begin{equation}\label{ineq:weak local monotonicity 1-2}
    \mathrm{VI}\leq 4\mathrm{II}+\frac{n}{\tau}\mathcal{E}+\mathrm{I}
    \end{equation}
    and
    \begin{equation}\label{ineq:weak local monotonicity 1-3}
    \mathrm{V}\leq 2\mathrm{II}+\tfrac12\mathrm{VI}
    \leq4\mathrm{II}+\frac{n}{2\tau}\mathcal{E}+\tfrac12\mathrm{I}.
    \end{equation}
    Note that $|X|\geq\frac18$ in $\spt|\nabla\phi|$, and thus
    \[
    \varrho\leq C(n) \tau^{-\frac n2}e^{-\frac{1}{256\tau}}.
    \]
    By the properties of $\phi$ and the assumption on total curvature, 
    \begin{equation}\label{ineq:weak local monotonicity 1-4}
    \mathrm{IV}
    \leq C(n)e^{-4ct}\int_{\spt|\nabla\phi|}|A|^n\varrho
    \leq C(n)\delta^n\tau^{-\frac n2}e^{-\frac{1}{256\tau}}.
    \end{equation}
    It remains to estimate $\mathrm{III}$.
    \begin{enumerate}[label=\textbf{Case \arabic*:}, leftmargin=*]
    \item $n \geq 3$. \\
    By H\"{o}lder's inequality and the Michael-Simon Sobolev inequality,  
    \[
    \begin{aligned}
    \mathrm{III}
    &\leq \Bigl(\int_{\spt\phi} |A|^n \Bigr)^{\frac{2}{n}}\Bigl(\int_{M_t}(\phi^2|A|^n\varrho)^{\frac{n}{n-2}}\Bigr)^{\frac{n-2}{n}}\\
    &\leq \delta^2\Bigl(\int_{M_t}(\phi^2|A|^n\varrho)^{\frac{n}{n-2}}\Bigr)^{\frac{n-2}{n}}\\
    &\leq C(n)\delta^2\Bigl(\int_{M_t}\left|\nabla(\phi|A|^{n/2}\varrho^{1/2})\right|^2
    +\int_{M_t}\phi^2|A|^n\varrho|\mathbf H|^2\Bigr)\\
    &\leq C(n)\delta^2(\mathrm{V}+n\mathrm{III})
    \end{aligned}
    \]
    For $\delta$ sufficiently small, we have
    \[
    \mathrm{III}\leq C(n)\delta^2\mathrm{V}.
    \]
    \item $n=2$.\\
    By Michael-Simon Sobolev inequality, H\"{o}lder's inequality and the identity 
    \[
    \nabla(fg^2)+g^2\nabla f=2g\nabla(fg),
    \]
    we have 
    \[
    \begin{aligned}
    \mathrm{III}
    \leq& C(n)\Bigl(\int_{M_t} \bigl\lvert\nabla(\phi|A|^2\varrho^{1/2})\bigr\rvert\Bigr)^2+C(n)\Bigl(\int_{M_t}\phi|A|^2\varrho^{1/2}|\mathbf H|\Bigr)^2\\
    \leq& C(n)\Bigl(\int_{M_t}\bigl\lvert2|A|\nabla(\phi|A|\varrho^{1/2})-|A|^2\nabla(\phi\varrho^{1/2})\bigr\rvert\Bigr)^2+C(n)\Bigl(\int_{M_t}\phi|A|^3\varrho^{1/2}\Bigr)^2\\
    \leq& C(n)\Bigl(\int_{M_t}|A|\cdot|\nabla(\phi|A|\varrho^{1/2})|\Bigr)^2
    +C(n)\Bigl(\int_{M_t} |A|^2|\nabla\phi|\varrho^{1/2}\Bigr)^2\\
    &+C(n)\Bigl(\int_{M_t}|A|^2\phi|\nabla\varrho^{1/2}|\Bigr)^2
    +C(n)\Bigl(\int_{M_t}\phi|A|^3\varrho^{1/2}\Bigr)^2\\
    \leq& C(n)\delta^2(\mathrm{V}+\mathrm{IV}+\mathrm{VI}+\mathrm{III}).
    \end{aligned}
    \]
    For $\delta$ sufficiently small, we have
    \[
    \mathrm{III}\leq C(n)\delta^2(\mathrm{V}+\mathrm{IV}+\mathrm{VI}).
    \]
    \end{enumerate}
    In both cases, for $t\in[0,t_0\wedge\frac{7}{128n}]$, 
    \begin{equation}\label{ineq:weak local monotonicity 1-5}
       \mathrm{III}\leq C(n)\delta^2(\mathrm{V}+\mathrm{IV}+\mathrm{VI}).
    \end{equation}
    Combining the estimates \eqref{ineq:weak local monotonicity 1-1}--\eqref{ineq:weak local monotonicity 1-5}, we obtain, for $t\in[0,t_0\wedge\frac{7}{128n}]$,
    \[
    \begin{aligned}
    \mathcal{E}'
    &\leq -\mathrm{I}-C(n)\mathrm{II}+C(n,m)\delta^2(\mathrm{V}+\mathrm{VI})+C(n)\mathrm{IV}\\
    &\leq -\mathrm{I}-C(n)\mathrm{II}+C(n,m)\delta^2(\mathrm{II}+\tau^{-1}\mathcal{E}+\mathrm{I})+C(n)\mathrm{IV}\\
    &\leq C_1\delta^2\tau^{-1}\mathcal{E}+C(n)\delta^n\tau^{-\frac n2}e^{-\frac{1}{256\tau}},
    \end{aligned}
    \]
    if $\delta$ is sufficiently small, where $C_1=C_1(n,m)$.
    Note that $\tau=2t_0-t\in[t_0,2t_0]$, and that
    \[
    \exp\left(C_1\delta^2\int_0^{t_0}\tau^{-1}dt\right)=2^{C_1\delta^2}.
    \]
    By Gr\"{o}nwall inequality, we have
    \[
    \mathcal{E}(t)\leq 2^{C_1\delta^2}\left[\mathcal{E}(0)+C(n)\delta^nt_0^{-\frac n2}e^{-\frac{1}{512 t_0}}\right],\quad\forall\,t\in[0,t_0\wedge\tfrac{7}{128n}].
    \]
    If $t_0\leq \frac{1}{128}$, then $\sqrt\tau\leq\frac18$, and hence
    \[
    e^{-8ct}\tau^{-\frac n2}e^{-\frac{1}{4}}\int_{B_{\sqrt\tau}(y)\cap M_{t}}|A|^n
    \leq 2^{C_1\delta^2}\biggl[\tau(0)^{-\frac n2}\int_{B_{3/4}(x_0)\cap M_0}|A|^n e^{-\frac{|\mathbf x_0-y|^2}{8t_0}}+C(n)\delta^nt_0^{-\frac n2}e^{-\frac{1}{512 t_0}}\biggr],
    \]
    that is, 
    \begin{equation}\label{ineq:weak local monotonicity 1-6}
    \int_{B_{\sqrt\tau}(y)\cap M_{t}}|A|^n
    \leq C(n,m)\biggl[\int_{B_{3/4}(x_0)\cap M_0}|A|^n e^{-\frac{|\mathbf x_0-y|^2}{8t_0}}+\delta^n e^{-\frac{1}{512 t_0}}\biggr].
    \end{equation}
    In particular, we have
    \[
    \|A\|_{L^n(B_{\sqrt{t_0}}(y)\cap M_t)}
    \leq C(n,m)\Bigl[\|A\|_{L^n(B_1(x_0)\cap M_0)}+\delta e^{-\frac{1}{512n t_0}}\Bigr].
    \]
    \end{proof}

The monotonicity formula propagates local volume ratio upper bounds to positive time.

\begin{lemma}\label{lem:volume ratio by initial}
    There exists a constant $T_2=T_2(n)$ such that the following holds.
    Fix $x_0\in\R^{n+m}$, $\rho>0$ and $\nu>0$.
    Let $\{M_t^n\}_{t\in[0,T]}$ be a smooth mean curvature flow in $\R^{n+m}$ satisfying
    \[
    B_{\rho}(x_0)\cap{M}_t\Subset{M}_t,\quad\forall\, t\in[0,T].
    \]
    If 
    \[
    \mu_0\left(B_r(x)\cap M_0\right)\leq \nu r^n,\quad\forall\,x\in B_{\rho/2}(x_0)\cap M_0,\quad\forall\,r\in[0,\rho/2],
    \]
    then
    \[
    \mu_t\left(B_r(y)\cap M_t\right)\leq C(n)\nu r^n,\quad\forall\,t\in[0,T\wedge T_2\rho^2],\quad\forall\,y\in B_{\rho/16}(x_0),\quad\forall\,r\in[0,\rho/16].
    \]
\end{lemma}
\begin{proof}
    It suffices to prove the case $\rho=1$.
    Define the measure $\lambda:=\mu_0\llcorner B_{1/4}(x_0)$.
    We claim that 
    \begin{equation}\label{ineq:volume ratio by initial-1}
    \lambda(B_R(y))\leq 2^n\nu R^n,\quad\forall\,y\in\R^{n+m},\quad\forall\,R>0.
    \end{equation}
    In fact, if $R\leq\frac14$, then for some $z\in B_R(y)\cap B_{1/4}(x_0)\cap M_0$, $B_{R}(y)\subset B_{2R}(z)$ and thus 
    \[
    \lambda(B_R(y))\leq \mu_0(B_{2R}(z)\cap M_0)\leq 2^n \nu R^n.
    \]
    If $R>\frac14$, then for some $z\in B_{1/4}(x_0)\cap M_0$, $B_{1/4}(x_0)\subset B_{1/2}(z)$ and thus 
    \[
    \lambda(\R^{n+m})\leq \mu_0(B_{1/2}(z)\cap M_0)\leq \nu 2^{-n}\leq 2^n\nu R^n.
    \]
    This proves \eqref{ineq:volume ratio by initial-1}.
    Therefore, for any $y\in\R^{n+m}$ and any $\sigma>0$, 
    \begin{equation}\label{ineq:volume ratio by initial-2}
    \begin{aligned}
    &(4\pi\sigma^2)^{-\frac n2}\int_{\R^{n+m}} e^{-\frac{|x-y|^2}{4\sigma^2}}\,d\lambda(x)\\
    \leq& (4\pi\sigma^2)^{-\frac n2}\left[\lambda(B_{\sigma}(y))+\sum_{k=1}^{\infty} e^{-4^{k-2}}\lambda(B_{2^k\sigma}(y))\right]\\
    \leq& 2^n(4\pi)^{-\frac n2}\left[1+\sum_{k=1}^{\infty} 2^{nk}e^{-4^{k-2}}\right]\nu\\
    \leq& C(n)\nu.
    \end{aligned}
    \end{equation}
    Set
    \[
    \eta(x,s)=|\mathbf x_s-x_0|^2+2ns,\quad \chi(x,s)=\left(1-16\eta(x,s)\right)_+^3,
    \]
    satisfying
    \begin{enumerate}[label=(\roman*)]
        \item $\spt\chi(\cdot,s)\subset \overline{B}_{1/4}(x_0)\cap M_s$; 
        \item $\chi\geq\frac18$ in $B_{1/8}(x_0)\cap M_s$ for $s\leq\frac{1}{128n}$;
        \item $\square\chi=-6\cdot 16^2\left(1-16\eta\right)_+|\nabla\eta|^2\leq0$.
    \end{enumerate}
    Fix
    \[
    t\in [0,T\wedge \tfrac{1}{128n}], \quad y\in B_{1/16}(x_0),\quad r\in(0,\tfrac{1}{16}].
    \]
    Define 
    \[
    \varrho(x,s)=(4\pi(t+r^2-s))^{-\frac n2}\exp\left(-\frac{|\mathbf x_s-y|^2}{4(t+r^2-s)}\right),\quad s\in[0,t].
    \]
    By Huisken's monotonicity formula, 
    \[
    \frac{d}{ds}\int_{M_s}\chi\varrho \,d\mu_s
    = \int_{M_s}\square\chi\varrho \,d\mu_s-\int_{M_s}\chi\left|\mathbf H+\frac{(\mathbf x_s-y)^\perp}{2(t+r^2-s)}\right|^2\varrho \,d\mu_s
    \leq0.
    \]
    At $s=0$, taking $\sigma=\sqrt{t+r^2}$ in \eqref{ineq:volume ratio by initial-2} yields 
    \[
    \int_{M_0}\chi\varrho \,d\mu_0\leq (4\pi\sigma^2)^{-\frac n2}\int e^{-\frac{|x-y|^2}{4\sigma^2}}\,d\lambda(x)\leq C(n)\nu.
    \]
    At $s=t$, since $B_r(y)\subset B_{1/8}(x_0)$,
    \[
    \int_{M_t}\chi\varrho \,d\mu_t\geq\tfrac18(4\pi r^2)^{-\frac{n}{2}}e^{-\frac14}\mu_t\left(B_r(y)\cap M_t\right).
    \]
    Combining these above gives 
    \[
    \mu_t\left(B_r(y)\cap M_t\right)\leq C(n)\nu r^n.
    \]
\end{proof}

Locally, under suitable initial assumptions and a small $t^{-1}$ curvature bound, we establish the weak monotonicity of the total curvature as follows.

\begin{lemma}\label{lem:weak local monotonicity 2}
    For any $\nu>0$ there exist constants $T_3=T_3(n)$, $\delta_2=\delta_2(n,m)$, $K=K(n,m)$ and $\alpha_1=\alpha_1(n,m,\nu)$ such that the following holds.
    Fix $x_0\in\R^{n+m}$, $\rho>0$ and $\delta\in(0,\delta_2]$. 
    Let $\{M_t^n\}_{t\in[0,T]}$ be a smooth mean curvature flow in $\R^{n+m}$ satisfying
    \begin{enumerate}[label=(\roman*)]
        \item $B_{\rho}(x_0)\cap{M}_t\Subset{M}_t,\quad\forall\,t\in[0,T]$;
        \item $\|A\|_{L^n(B_{\rho}(x_0)\cap{M}_0)}\leq \delta$;
        \item $\sup_{B_{\rho}(x_0)\cap{M}_t}|A|^2\leq \alpha_1\delta^2 t^{-1}, \quad\forall\,t\in(0,T]$;
        \item $\mu_0\left(B_r(x)\cap M_0\right)\leq \nu r^n,\quad\forall\,x\in B_{\rho/2}(x_0)\cap M_0,\quad\forall\,r\in[0,\rho/2].$
    \end{enumerate}
    Then
    \[
    \|A\|_{L^n(B_{\rho/32}(x_0)\cap{M}_t)}\leq K\delta,
    \quad\forall\,t\in[0,T\wedge T_3\rho^2].
    \]
\end{lemma}
\begin{proof}
    It suffices to prove the case $\rho=1$.
    For $x\in B_{1/16}(x_0)$, define $\rho(x)=\frac{1}{2}(\frac{1}{16}-|x-x_0|)$.
    Set $T_3=T_1\wedge T_2$, where $T_1=T_1(n)$ is from Lemma~\ref{lem:weak local monotonicity 1} and $T_2=T_2(n)$ is from Lemma~\ref{lem:volume ratio by initial}. 
    We claim that for $K$ to be determined,
    \[
    \|A\|_{L^n(B_{\rho(x)}(x)\cap{M}_t)}\leq K\delta,
    \quad \forall\,(x,t)\in B_{1/16}(x_0)\times[0,T\wedge T_3].
    \]
    Otherwise, there exists $\bar{t}\in(0,T\wedge T_3]$ such that
    \begin{enumerate}[label=(\Roman*)]
        \item $\|A\|_{L^n(B_{\rho(x)}(x)\cap {M}_t)}\leq K\delta$ for any $(x,t)\in B_{1/16}(x_0)\times[0,\bar{t})$;
        \item $\|A\|_{L^n(B_{\rho(\bar{x})}(\bar{x})\cap {M}_{\bar{t}})}=K\delta$ for some $\bar{x}\in B_{1/16}(x_0)$.
    \end{enumerate}
    Set $\bar{\rho}:=\rho(\bar{x})\in(0,\frac{1}{32}]$.
    By (II) and Lemma~\ref{lem:volume ratio by initial}, 
    \[
    \begin{aligned}
    K\delta
    \leq \sqrt{\alpha_1\delta^2\bar{t}^{-1}}\cdot \mu_{\bar{t}}^{\frac1n}(B_{\bar{\rho}}(\bar{x})\cap {M}_{\bar{t}})
    \leq \sqrt{\alpha_1\delta^2\bar{t}^{-1}}\cdot C_1\nu^{\frac1n}\bar{\rho},
    \end{aligned}
    \]
    where $C_1=C_1(n)$.
    It follows that
    \begin{equation}\label{ineq:weak local monotonicity 2-1}
    \frac{\bar{\rho}^2}{\bar{t}}\geq C_1^{-2} K^2\alpha_1^{-1}\nu^{-\frac2n}.
    \end{equation}
    For $x\in B_{\bar{\rho}}(\bar{x})\subset B_{1/16}(x_0)$, 
    \[
    \begin{aligned}
    2\rho(x)
    &=\tfrac{1}{16}-|x-x_0|
    \geq \tfrac{1}{16}-|\bar{x}-x_0|-|\bar{x}-x|
    \geq 2\bar{\rho}-\bar{\rho}
    =\bar{\rho}.
    \end{aligned}
    \]
    By (I) and (II), 
    \[
    \|A\|_{L^n(B_{\bar{\rho}/2}(x)\cap {M}_t)}\leq K\delta,
    \quad \forall\,(x,t)\in B_{\bar{\rho}}(\bar{x})\times[0,\bar{t}].
    \]
    By \eqref{ineq:weak local monotonicity 1-6} from the proof of Lemma~\ref{lem:weak local monotonicity 1}, if $\bar{t}\leq T_1(\frac{\bar{\rho}}2)^2$ and $K\delta_2\leq \delta_1=\delta_1(n,m)$, then for any $x\in B_{\bar{\rho}}(\bar{x})$ and $t=\bar{t}$ we have
    \[
    \begin{aligned}
    \int_{B_{\sqrt{\bar{t}}}(x)\cap {M}_{\bar{t}}}|A|^n
    &\leq C_2\biggl[\int_{B_{\bar{\rho}/2}(x)\cap {M}_0}|A|^n e^{-\frac{|\mathbf x_0-x|^2}{8\bar{t}}}+(K\delta)^n e^{-\frac{\bar{\rho}^2}{2048\bar{t}}}\biggr]\\
    &\leq C_2\biggl[\int_{B_{\rho(x)}(x)\cap {M}_0}|A|^n e^{-\frac{|\mathbf x_0-x|^2}{8\bar{t}}}+(K\delta)^n e^{-\frac{\bar{\rho}^2}{2048\bar{t}}}\biggr],
    \end{aligned}
    \]
    where $C_2=C_2(n,m)$. 
    By \eqref{ineq:weak local monotonicity 2-1}, $\bar\rho\geq\sqrt{\bar t}$ for sufficiently small $\alpha_1$.
    Choose a maximal $2\sqrt{\bar t}/5$-separated set $\{x_i\}_{i=1}^N\subset B_{\bar{\rho}-\sqrt{\bar t}/5}(\bar{x})$ such that
    \[
    \bigsqcup_{i=1}^N B_{\sqrt{\bar{t}}/5}(x_i)
    \subset B_{\bar{\rho}}(\bar{x})
    \subset\bigcup_{i=1}^N B_{\sqrt{\bar{t}}}(x_i)
    \]
    with
    \[
    N\leq 5^{n+m}\left(\frac{\bar{\rho}}{\sqrt{\bar{t}}}\right)^{n+m}.
    \]
    By (II) and the preceding estimate,
    \begin{equation}\label{ineq:weak local monotonicity 2-2}
    \begin{aligned}
    (K\delta)^n
    \leq& C_2 \sum_i \biggl[\int_{B_{\rho(x_i)}(x_i)\cap {M}_0}|A|^n e^{-\frac{|\mathbf x_0-x_i|^2}{8\bar{t}}}+(K\delta)^n e^{-\frac{\bar{\rho}^2}{2048\bar{t}}}\biggr]\\
    \leq& C_2 \biggl[\int_{B_{1/8}(x_0)\cap {M}_0}|A|^n \sum_ie^{-\frac{|\mathbf x_0-x_i|^2}{8\bar{t}}}+5^{n+m}(K\delta)^n\left(\tfrac{\bar{\rho}^2}{\bar{t}}\right)^{\frac{n+m}{2}} e^{-\frac{\bar{\rho}^2}{2048\bar{t}}}\biggr].
    \end{aligned}
    \end{equation}
    Note that for any $z\in\R^{n+m}$, if
    \[
    i\in I_k:=\bigl\{\,1\leq i\leq N \mid x_i\in A_{k\sqrt{\bar t},(k+1)\sqrt{\bar t}}(z)\,\bigr\},
    \]
    then 
    \[
    B_{\sqrt{\bar t}/5}(x_i)\subset B_{(k+2)\sqrt{\bar t}}(z).
    \]
    It follows that
    \begin{equation}\label{ineq:weak local monotonicity 2-3}
    \begin{aligned}
    \sum_ie^{-\frac{|z-x_i|^2}{8\bar t}}
    \leq \sum_{k=0}^\infty\sum_{i\in I_k}e^{-\frac{|z-x_i|^2}{8\bar t}}
    \leq \sum_{k=0}^\infty 5^{n+m}(k+2)^{n+m}e^{-\frac{k^2}{8}}
    \leq C(n,m).
    \end{aligned}
    \end{equation}
    On the other hand, since 
    \[
    \lim_{x\to\infty}x^{\frac{n+m}{2}}e^{-\frac{x}{2048}}=0,
    \]
    there exists a constant $C_3=C_3(n,m)$ such that if $\tfrac{\bar{\rho}^2}{\bar{t}}\geq C_3$, then
    \begin{equation}\label{ineq:weak local monotonicity 2-4}
    (\tfrac{\bar{\rho}^2}{\bar{t}})^{\frac{n+m}{2}} e^{-\frac{\bar{\rho}^2}{2048\bar{t}}}\leq \tfrac12 5^{-n-m}C_2^{-1}.
    \end{equation}
    By \eqref{ineq:weak local monotonicity 2-1}, it suffices to ensure
    \[
    K^2\alpha_1^{-1}\geq \max\{4T_1^{-1},C_3\}C_1^{2}\nu^{\frac2n}.
    \]
    Combining \eqref{ineq:weak local monotonicity 2-2}--\eqref{ineq:weak local monotonicity 2-4}, we obtain
    \[
    \tfrac12(K\delta)^n\leq C_4\delta^n
    \]
    for some $C_4=C_4(n,m)$.
    Choosing 
    \[
    K> (2C_4)^{\frac1n},\quad \alpha_1\leq \min\{{T_1}/4,C_3^{-1}\}C_1^{-2}\nu^{-\frac2n}K^2
    \]
    yields a contradiction.
\end{proof}

\subsection{Curvature improvement}$\\$

We use the following localized version of the $\epsilon$-regularity estimate in \cite{MR3864377}.
\begin{lemma}\label{lem:epsilon-regularity}
    For $p\in[n,n+2]$, there exists a constant $\epsilon_0=\epsilon_0(n,m,p)>0$ such that the following holds.
    Let $\mathbf x:M\times(-4,0]\to\R^{n+m}$ be a smooth mean curvature flow.
    Fix $t_0\in[-1,0]$ and $x_0\in M_{t_0}$.
    Suppose that there is a compact set $K\Subset M^\circ$ such that 
    \[
    \mathbf x_t^{-1}(B_2(x_0))\subset K,\quad\forall\, t\in(-4,0].
    \]
    If 
    \[
    I_{x_0,t_0}
    :=\sup_{\sqrt{-t_0}\leq\rho<\rho'\leq 2}\frac{1}{((\rho')^2-\rho^2)^{\frac{n+2-p}{2}}}\int_{-(\rho')^2}^{-\rho^2}\int_{ B_2(x_0)\cap M_t}|A|^p \,d\mu_t dt
    \leq \epsilon_0,
    \]
    then 
    \[
    \sup_{\sigma\in(0,1)}\left[\sigma^2\sup_{t\in(t_0-(1-\sigma)^2,t_0)}\sup_{ B_{1-\sigma}(x_0)\cap M_t} |A|^2\right]
    \leq (\epsilon_0^{-1}I_{x_0,t_0})^{\frac2p}.
    \]
    In particular, 
    \[
    |A|^2(x_0,t_0)\leq 4(\epsilon_0^{-1}I_{x_0,t_0})^{\frac2p}.
    \]
\end{lemma}

We now combine the preceding energy estimates with local $\epsilon$-regularity to improve the preliminary $t^{-1}$ curvature bound. 
This improvement is the main input for the local existence construction in the next section.

\begin{proposition}[Curvature improvement]\label{prop:time-decay estimate}
    For any $\nu>0$ there exist constants $\delta_3=\delta_3(n,m)$, $\alpha_2=\alpha_2(n,m,\nu)$, $\alpha_3=\alpha_3(n,m)$ and $T_4=T_4(n,m,\nu)$  such that the following holds.
    Fix $x_0\in\R^{n+m}$, $\rho>0$ and $\delta\in(0,\delta_3]$.
    Let $\{M_t^n\}_{t\in[0,T]}$ be a smooth mean curvature flow in $\R^{n+m}$ such that
    \begin{enumerate}[label=(\roman*)]
        \item $B_{\rho}(x_0)\cap {M}_0\Subset {M}_0$;
        \item $\|A\|_{L^n(B_{3\rho/4}(x_0)\cap {M}_0)}\leq\delta$;
        \item $\sup_{B_{\rho}(x_0)\cap {M}_t}|A|^2\leq \alpha_2 t^{-1}, \quad\forall\,t\in(0,T]$;
        \item $\mu_0\left(B_r(x)\cap M_0\right)\leq \nu r^n,\quad\forall\,x\in B_{3\rho/4}(x_0)\cap M_0,\quad\forall\,r\in[0,\rho/4].$
    \end{enumerate}
    Then
    \[
    \sup_{B_{\rho/2}(x_0)\cap M_t}|A|^2
    \leq\alpha_3\left[\delta^2+e^{-\frac{\rho^2}{2^{16}nt}}\right]t^{-1},
    \quad\forall\,t\in(0,T\wedge T_4\rho^2].
    \]
    In particular, there exists a constant $T_5=T_5(n,m,\nu)$ such that
    \[
    \sup_{B_{\rho/2}(x_0)\cap M_t}|A|^2
    \leq 2\alpha_3 \delta^2 t^{-1},
    \quad\forall\,t\in\left(0,T\wedge T_5|\log\delta|^{-1}\rho^2\right].
    \]
\end{proposition}
\begin{proof}
    It suffices to prove the case $\rho=1$.
    For $t\leq \frac{1}{64n\alpha_2}$ and $q_t=\mathbf x_t(q)\in B_{3/4}(x_0)\cap M_t$, by Lemma~\ref{lem:shrinking ball}, 
    \[
    B_{1/4-2\sqrt{n\alpha_2 t}}(q_0)\subset B_{1/4}(q_t)\subset B_{1}(x_0).
    \]
    Thus, 
    \[
    B_{3/4}(x_0)\cap M_t
    \subset \Phi_{0,t}\left(B_{1}(x_0)\cap M_0\right)
    \Subset \Phi_{0,t}(M_0)=M_t,
    \quad\forall\,t\in [0,T\wedge\tfrac{1}{64n\alpha_2}].
    \]
    By choosing $\{y_i\}_{i=1}^N\subset B_{9/16}(x_0)$ we consider the covering 
    \[
    \bigsqcup_{i=1}^N B_{1/1280}(y_i)
    \subset B_{9/16}(x_0)
    \subset\bigcup_{i=1}^N B_{1/256}(y_i)
    \]
    with $N\leq N_0(n,m)$.
    Note that $B_{1/16}(y_i)\subset B_{1/8}(y_i)\subset B_{3/4}(x_0)$.
    By Lemma~\ref{lem:weak local monotonicity 2}, if $\delta_3\leq\delta_2\wedge K^{-1}N_0^{-\frac1n}\delta_1$ and $\alpha_2\leq\alpha_1\delta_3^2$, then for each $y_i$ and $t\leq T\wedge\frac{1}{64n\alpha_2}\wedge\tfrac{T_3}{64}$, 
    \[
    \|A\|_{L^n(B_{1/256}(y_i)\cap M_t)}\leq K\delta_3\leq N_0^{-\frac1n}\delta_1,
    \]
    and thus
    \[
    \|A\|_{L^n(B_{9/16}(x_0)\cap M_t)}\leq N_0^{\frac1n}K\delta_3\leq \delta_1.
    \]
    Fix $t\leq T\wedge \frac{1}{64n\alpha_2}\wedge\frac{T_3}{64}\wedge\frac{T_1}{256}\wedge\frac14$ and $y_t\in B_{1/2}(x_0)\cap M_t$ so that $B_{1/16}(y_t)\subset B_{9/16}(x_0)$.
    By Lemma~\ref{lem:weak local monotonicity 1}, for any $s\in[\frac34 t,t]$, 
    \[
    \begin{aligned}
    \|A\|_{L^n(B_{\sqrt t/2}(y_t)\cap M_s)}
    &\leq \|A\|_{L^n(B_{{\sqrt s}}(y_t)\cap M_s)}\\
    &\leq C(n,m)\Bigl[\|A\|_{L^n(B_{1/16}(y_t)\cap M_0)}+\delta_1 e^{-\frac{1}{2^{17}ns}}\Bigr]\\
    &\leq C(n,m)\bigl[\delta+e^{-\frac{1}{2^{17}nt}}\bigr].
    \end{aligned}
    \]
    Consider the rescaled flow 
    \[
    \tilde{M}_{\tau}=\tfrac{4}{\sqrt{t}}\left(M_{t+\frac{t}{16}\tau}-y_t\right),
    \quad \tau\in[-4,0].
    \]
    Note that $B_{\sqrt t/2}(y_t)\subset B_{3/4}(x_0)$ and 
    \[
    \mathbf x_s^{-1}(B_{3/4}(x_0))\subset \overline{\mathbf x_0^{-1}(B_1(x_0))},\quad\forall\,s\in[0,t].
    \]
    Thus the geometric hypotheses of Lemma~\ref{lem:epsilon-regularity} are satisfied with $p=n$, $t_0=0$ and $x_0=0$.
    If 
    \[
    \begin{aligned}
    I_{0,0}
    &=\sup_{0\leq a< b\leq 4}\frac{1}{b-a}\int_{-b}^{-a}\int_{B_2\cap\tilde{ {M}}_\tau}|\tilde{A}|^n \,d\tilde{\mu}_\tau \,d\tau\\
    &=\sup_{\frac34 t\leq u< v\leq t}\frac{1}{v-u}\int_u^v\int_{B_{\sqrt{t}/2}(y_t)\cap {M}_s}|A|^n \,d\mu_s ds\\
    &\leq C_1\left[\delta^n+e^{-\frac{1}{2^{17}t}}\right]
    \leq \epsilon_0,
    \end{aligned}
    \]
    where $C_1=C_1(n,m)$ and $\epsilon_0=\epsilon_0(n,m)$, then 
    \[
    \frac{t}{16}|A|^2(y_t,t)=|\tilde{A}|^2(0,0)
    \leq 4(\epsilon_0^{-1}I_{0,0})^{\frac{2}{n}}
    \leq C_2\left[\delta^2+e^{-\frac{1}{2^{16}nt}}\right]
    \]
    for some $C_2=C_2(n,m)$. 
    Hence, for 
    \[
    \delta_3\leq \min\left\{\delta_2,K^{-1}N_0^{-\frac1n}\delta_1,\tfrac12,\left(\frac{\epsilon_0}{2C_1}\right)^{\frac1n}\right\}
    \] 
    and 
    \[
    t\leq T\wedge \min\left\{\frac{1}{64n\alpha_2}, \frac{T_3}{64}, \frac{T_1}{256}, \frac14,\frac{1}{2^{17}|\log({\epsilon_0}/(2C_1))|}\right\},
    \] 
    we have
    \[
    \sup_{B_{1/2}(x_0)\cap M_t}|A|^2
    \leq 16C_2\left[\delta^2+e^{-\frac{1}{2^{16}nt}}\right]t^{-1}.
    \]
    Since $\delta_3\leq\tfrac12$, the final assertion follows by taking
    \[
    T_5:=\min\left\{T_4\log 2,\frac{1}{2^{17}n}\right\}.
    \]
\end{proof}

 \section{Local existence and volume estimates}\label{sec4}

In this section, we use the estimates from Section~\ref{sec3} to construct local mean curvature flows from initial submanifolds with small total curvature. 
We then use these flows to prove a quantitative volume estimate for local branches of the initial submanifold.

\subsection{Construction of local mean curvature flow}$\\$

Following the iterative argument in Theorem~1.1 of \cite{MR4854457}, we prove the following short-time existence.

\begin{theorem}[Local existence under small total curvature]\label{thm:local short-time existence by small total curvature}
    For any $\nu>0$ there exist constants $\delta_4=\delta_4(n,m,\nu)$ and $T_6=T_6(n,m,\nu)$ such that the following holds.
    Let $R>0$ and let ${M}_0^n\subset\R^{n+m}$ be a submanifold properly immersed in $B_{R+4}$ with $0\in {M}_0$. 
    If $M_0$ satisfies the total curvature bound
    \[
    \|A\|_{L^n(B_{R+4}\cap {M}_0)}\leq \delta_4
    \]
    and the volume ratio bound
    \[
    \sup_{\substack{x\in B_{R+3}\cap M_0\\0<r\leq 1}} r^{-n}\mu_0(B_r(x)\cap M_0)\leq \nu,
    \]
    then there exists a smooth mean curvature flow $\{M_t^n\}_{t\in[0,T_6]}$ properly immersed in $B_{R+1}$ with initial data $B_{R+1}\cap M_0$ and nonempty time slices such that 
    \[
    \sup_{B_{R+1}\cap {M}_t}|A|^2\leq \alpha_2 t^{-1},\quad\forall\,t\in(0,T_6].
    \]
    Here, $\alpha_2=\alpha_2(n,m,\nu)$ is from Proposition~\ref{prop:time-decay estimate}.
\end{theorem}
\begin{proof}
    For dimension $n\geq2$ and codimension $m\geq1$, we use the following constants:
    \begin{itemize}
        \item $\Gamma=\Gamma(n,m)$ and $\Lambda=\Lambda(n,m)$ from Proposition~\ref{prop:local short-time existence};
        \item $\delta_3=\delta_3(n,m)$, $\alpha_2=\alpha_2(n,m,\nu)$, $\alpha_3=\alpha_3(n,m)$ and $T_5=T_5(n,m,\nu)$ from Proposition~\ref{prop:time-decay estimate};
        \item $\delta_4=\delta_4(n,m,\nu):=\min\{\delta_3,e^{-1},\frac12\Lambda^{-\frac12}\alpha_3^{-\frac12}\alpha_2^{\frac12},\Gamma^{\frac12}(2\alpha_3)^{-\frac12}\}$;
        \item $L=L(n,m,\nu):=\max\left\{2T_5^{-\frac12}\alpha_2^{-\frac12}|\log\delta_4|^{\frac 12}, 2\sqrt2\alpha_3^{-\frac12}\alpha_2^{-\frac12}\delta_4^{-1}\right\}$.
    \end{itemize}
    Set 
    \[
    \rho:=\min\Bigl\{\tfrac14,\bigl(\sup_{B_{R+\frac72}\cap {M}_0}|A|\bigr)^{-1}\Bigr\}>0.
    \]
    Define
    \[
    t_1=\min\{\Gamma\rho^2,\Lambda^{-1}\alpha_2\rho^2\},\qquad 
    r_1=R+3,
    \]
    and, recursively, 
    \[
    t_{i+1}=2t_i,\qquad 
    r_{i+1}=r_i-L\sqrt{\alpha_2 t_i}.
    \]
    Choose $l$ such that $r_{l+1}< R+1\leq r_l$.
    %We first obtain a lower bound for $t_l$.
    By the recursive definition,
    \[
    2
    <r_1-r_{l+1}
    =L\sqrt{\alpha_2}\sum_{j=1}^l\sqrt{t_j}
    \leq L\sqrt{\alpha_2 t_l}\sum_{j=0}^\infty 2^{-\frac j2}
    \leq \frac{\sqrt2}{\sqrt2-1}L\sqrt{\alpha_2t_l},
    \]
    which implies
    \[
    t_l\geq 2(\sqrt2-1)^2 L^{-2}\alpha_2^{-1}=:T_6'.
    \]
    
    For each $i\geq1$, let $\mathcal{P}_i$ be the following assertion:
    there exists a smooth mean curvature flow $\{M_t\}_{t\in[0,t_i]}$ properly immersed in $B_{r_i}$, with initial data $B_{r_i}\cap M_0$, satisfying
    \[
    \sup_{B_{r_i}\cap M_t}|A|^2\leq \alpha_2 t^{-1},\quad\forall\,t\in(0,t_i].
    \]
    We claim that $\mathcal{P}_i$ holds for each $i\leq l$.
        
    Note that $\mathcal{P}_1$ holds by Proposition~\ref{prop:local short-time existence} since $\Lambda\rho^{-2}\leq\alpha_2 t_1^{-1}$.
    Assume that $l\geq2$ and $\mathcal{P}_{i}$ holds for some $i\leq l-1$.
    It suffices to prove $\mathcal{P}_{i+1}$.
    Set 
    \[
    \lambda_{i}:=\tfrac14 L\sqrt{\alpha_2 t_{i}},\qquad
    \rho_i:=(2\alpha_3)^{-\frac12}\delta_4^{-1}\sqrt t_i\leq\lambda_i,
    \]
    where the inequality follows from our choice of $L$.
    Fix $y\in B_{r_{i+1}+2\lambda_{i}}\cap {M}_{t_i}$. 
    Since $2\lambda_{i}=\frac12(r_{i}-r_{i+1})\leq1$, $B_{2\lambda_{i}}(y)\subset B_{r_{i}}\subset B_{R+3}$.
    By the assumption on volume ratio, 
    \[
    \sup_{\substack{x\in B_{2\lambda_{i}}(y)\cap M_0\\0<r\leq 2\lambda_{i}}} r^{-n}\mu_0(B_r(x)\cap M_0)\leq \nu.
    \]
    Our choice of $L$ gives 
    \[
    4T_5|\log\delta_4|^{-1}\lambda_{i}^2
    =\frac{T_5L^2\alpha_2}{4|\log\delta_4|}t_i
    \geq t_{i}.
    \]
    Since $\delta_4\leq\delta_3$, Proposition~\ref{prop:time-decay estimate} yields
    \[
    |A|^2(y,t_i)
    \leq 2\alpha_3\delta_4^2{t_i}^{-1}=\rho_i^{-2}.
    \]
    Hence, $M_{t_i}$ is properly immersed in $B_{r_{i+1}+2\rho_i}$ with 
    \[
    \sup_{B_{r_{i+1}+2\rho_{i}}\cap {M}_{t_{i}}}|A|^2
    \leq \rho_{i}^{-2}.
    \]
    By Proposition~\ref{prop:local short-time existence}, there exists a smooth mean curvature flow properly immersed in $B_{r_{i+1}}$ for $t\in[t_{i},t_{i}+\Gamma\rho_{i}^2]$, with initial data $B_{r_{i+1}}\cap M_{t_{i}}$.
    Identifying the initial parametrization of the restarted flow with that of the original flow at $t_i$, we concatenate the two flows in $B_{r_{i+1}}$.
    The concatenation is smooth across $t_i$, since the mean curvature flow equation determines all time derivatives there from the common initial immersion.
    This extends $\{M_t\}_{t\in[0,t_i]}$ to the time interval $[t_{i},t_{i+1}]$ since $t_{i}\leq \Gamma\rho_{i}^2$ by our choice of $\delta_4$.
    In addition, since $\Lambda\alpha_3\delta_4^2\leq \frac14\alpha_2$ by our choice of $\delta_4$,
    \[
    \sup_{B_{r_{i+1}}\cap M_t}|A|^2
    \leq \Lambda \rho_{i}^{-2}
    =2\Lambda\alpha_3 \delta_4^2 t_{i}^{-1}
    \leq \tfrac12\alpha_2t_{i}^{-1}
    \leq \alpha_2 t^{-1},
    \quad\forall\,t\in[t_{i},t_{i+1}].
    \]
    This proves the claim.
    
    Set
    \[
    T_6=T_6'\wedge \tfrac{1}{4n\alpha_2}.
    \]
    By Lemma~\ref{lem:shrinking ball} and Lemma~\ref{lem:distance distortion}, the forward trajectory $\{\Phi_{0,t}(0)\}_{t\geq0}$ issuing from $0\in M_0$ satisfies 
    \[
    |\Phi_{0,t}(0)|\leq 2\sqrt{n\alpha_2 t}\leq 1< R+1,\quad\forall\,t\in[0,T_6].
    \]
    Hence every time slice intersects $B_{R+1}$.
    Restricting the flow obtained from $\mathcal{P}_l$ to $B_{R+1}\times[0,T_6]$ completes the proof.
\end{proof}

A scale-invariant version also follows.
\begin{corollary}\label{cor:scale-invariant form of local short-time existence}
    For any $\nu>0$ there exist constants $\delta_5=\delta_5(n,m,\nu)$ and $T_7=T_7(n,m,\nu)$ such that the following holds.
    Let $\rho>0$ and let $M_0^n\subset\R^{n+m}$ be a submanifold properly immersed in $B_\rho$ with $0\in {M_0}$.
    If, for some $\delta\in(0,\delta_5]$, $M_0$ satisfies the total curvature bound
    \[
    \|A\|_{L^n(B_{\rho}\cap {M_0})}\leq \delta
    \]
    and the volume ratio bound
    \[
    \sup_{\substack{x\in B_{5\rho/6}\cap M_0\\0<r\leq\rho/6}} r^{-n}\mu_0(B_r(x)\cap M_0)\leq \nu,
    \]
    then there exists a smooth mean curvature flow $\{M_t\}_{t\in [0,T_7|\log\delta|^{-1}\rho^2]}$ properly immersed in $B_{\rho/4}$ with initial data $B_{\rho/4}\cap M_0$ and nonempty time slices such that 
    \[
    \sup_{B_{\rho/4}\cap M_t}|A|^2
    \leq 2\alpha_3 \delta^2 t^{-1},
    \quad\forall\,t\in(0,T_7|\log\delta|^{-1}\rho^2].
    \]
    Here, $\alpha_3=\alpha_3(n,m)$ is from Proposition~\ref{prop:time-decay estimate}.
\end{corollary}
\begin{proof}
    Set $\tilde{M}_0=\frac6\rho M_0$.
    By scaling, 
    \[
    \|\tilde{A}\|_{L^n(B_{6}\cap\tilde{M}_0)}=\|A\|_{L^n(B_\rho\cap M_0)}\leq\delta,
    \]
    and 
    \[
    \sup_{\substack{x\in B_5\cap\tilde{M}_0\\0<r\leq1}} r^{-n}\tilde{\mu}_0(B_r(x)\cap \tilde{M}_0)
    =\sup_{\substack{x\in B_{5\rho/6}\cap M_0\\0<r\leq\rho/6}} r^{-n}\mu_0(B_r(x)\cap M_0)
    \leq \nu.
    \]
    For $\delta\leq\delta_4$, by Theorem~\ref{thm:local short-time existence by small total curvature} there exists a smooth mean curvature flow $\{\tilde{M}_s\}_{s\in[0,T_6]}$ properly immersed in $B_3$ with initial data $B_3\cap\tilde{M}_0$ such that 
    \[
    \sup_{B_3\cap\tilde{M}_s}|\tilde{A}|^2\leq \alpha_2 s^{-1},
    \quad\forall\,s\in(0,T_6].
    \]
    For $\delta\leq \delta_3$, by Proposition~\ref{prop:time-decay estimate}, 
    \[
    \sup_{B_{3/2}\cap\tilde{M}_s}|\tilde{A}|^2\leq 2\alpha_3\delta^2 s^{-1},
    \quad\forall\,s\in(0,T_6\wedge 9T_5|\log\delta|^{-1}].
    \]
    As argued at the end of the proof of Theorem~\ref{thm:local short-time existence by small total curvature}, 
    \[
    |\tilde{0}_s|\leq 2\sqrt{n\alpha_2 s}\leq1<\tfrac32, \quad\forall\,0\leq s\leq T_6.
    \]
    Hence every time slice $B_{3/2}\cap \tilde{M}_s$ is nonempty.
    For $\delta\leq e^{-9T_5 T_6^{-1}}$, $9T_5|\log\delta|^{-1}\leq T_6$.
    Rescale the flow back by 
    \[
    M_t=\tfrac{\rho}{6} \left(\tilde{M}_{36\rho^{-2}t}\cap B_{3/2}\right).
    \]
    $\{M_t\}$ is a smooth mean curvature flow properly immersed in $B_{\rho/4}\times[0,\frac{T_5}4|\log\delta|^{-1}\rho^2]$ with initial data $B_{\rho/4}\cap M_0$ such that
    \[
    \sup_{B_{\rho/4}\cap M_t}|A|^2\leq 2\alpha_3\delta^2 t^{-1},\quad\forall\,t\in(0,\tfrac{T_5}4|\log\delta|^{-1}\rho^2].
    \]
    Choosing 
    \[
    \delta_5=\min\{\delta_4,\delta_3,\exp(-\tfrac{9T_5}{T_6})\}, \quad 
    T_7=\tfrac{T_5}4
    \]
    completes the proof.
\end{proof}

\subsection{Almost Euclidean volume ratio estimates}$\\$

We now use the local flow to estimate the volume of branches of the initial submanifold. 
At a suitable positive time, the curvature bound gives a graphical representation and an almost Euclidean volume ratio for the local branches. 
Estimates for displacement and volume change then transfer this control to the initial time.

\begin{theorem}\label{thm:volume ratio by small total curvature}
    For any $\nu>0$ there exist constants $\delta_6=\delta_6(n,m,\nu)$ and $C=C(n,m)$ such that the following holds.
    Let $\rho>0$ and let $\mathbf x_0:M^n\to\R^{n+m}$ be a smooth proper immersion in $B_\rho\subset\R^{n+m}$, with $0=\mathbf x_0(p)$ for some $p\in M$.
    Write $M_0=\mathbf x_0(M)$ and $g_0=\mathbf x_0^* g_e$.
    Suppose that 
    \[
    \|A\|_{L^n(B_\rho\cap M_0)}\leq\delta\leq \delta_6
    \]
    and 
    \[
    \sup_{\substack{x\in B_{5\rho/6}\cap M_0\\0<s\leq\rho/6}} s^{-n}\mu_{g_0}(B_s(x)\cap M_0)\leq\nu.
    \]
    Then, for any $0<r\leq \rho/256$ and 
    \[
    E_{p,r}:=\mathbf x_0^{-1}(B_r)\cap C_p(\mathbf x_0^{-1}(B_{2r})), 
    \]
    we have 
    \[
    (1-C\delta^{\frac43})\omega_n r^n
    \leq \mu_{g_0}(E_{p,r})
    \leq (1+C\delta^{\frac43})\omega_n r^n.
    \]
\end{theorem}
\begin{proof}
    It suffices to prove the case $\rho=1$ and $\delta>0$.
    
    \noindent\textbf{Step 1. Construction of the local flow.}

    Fix $r\in(0,\tfrac{1}{256}]$ and set
    \[
    N:=C_p(\mathbf x_0^{-1}(B_{8r})),\quad W:=C_p(\mathbf x_0^{-1}(B_{2r})),\quad E:=\mathbf x_0^{-1}(B_r)\cap W.
    \]
    For $\delta_6\leq\delta_5$, by Corollary~\ref{cor:scale-invariant form of local short-time existence} there exists a smooth mean curvature flow $\{M_t\}_{t\in[0,T_7|\log\delta|^{-1}]}$ properly immersed in $B_{1/4}$ starting from $B_{1/4}\cap M_0$
    such that
    \begin{equation}\label{ineq:volume ratio by small total curvature-1}
    \sup_{B_{1/4}\cap M_t}|A|^2\leq 2\alpha_3\delta^2 t^{-1},
    \quad \forall\,t\in(0,T_7|\log\delta|^{-1}].
    \end{equation}
    Along the flow, write
    \[
    W_t:=\mathbf x_t(W),\quad E_t:=\mathbf x_t(E).
    \]
    Let 
    \begin{equation}\label{ineq:volume ratio by small total curvature-2}
    t_*:=\theta r^2\in(0,T_7|\log\delta|^{-1}]
    \end{equation}
    with $\theta$ to be determined later.
    Assume that
    \begin{equation}\label{ineq:volume ratio by small total curvature-3}
    \theta<(800n\alpha_3)^{-1}\delta^{-2}.
    \end{equation}
    Set
    \begin{equation}\label{ineq:volume ratio by small total curvature-4}
     D:=2\sqrt{2n\alpha_3\delta^2 t_*}=2\sqrt{2n\alpha_3\theta}\delta r<r/10.
    \end{equation}
    By \eqref{ineq:volume ratio by small total curvature-1} and Lemma~\ref{lem:shrinking ball}, 
    \[
    B_{6r}(q_t)
    \subset B_{8r-D}(q_t)
    \subset B_{8r}(q_0) 
    \subset B_{1/4},
    \quad\forall\,(q,t)\in N\times[0,t_*].
    \]
    Hence, we have the curvature estimate
    \begin{equation}\label{ineq:volume ratio by small total curvature-5}
    \kappa^2r^2:=r^2\sup_{B_{6r}(p_{t_*})\cap M_{t_*}}|A|^2\leq 2\alpha_3\delta^2\theta^{-1}.
    \end{equation}
    Also, by Lemma~\ref{lem:distance distortion}, 
    \[
    |q_t-q_0|\leq D<r/10,\quad\forall\,(q,t)\in N\times[0,t_*].
    \]
    Let
    \[
    L:=\max\left\{1,\sqrt{\frac{2\alpha_3}{\alpha_1}}\right\}
    \]
    so that $2\alpha_3\delta^2\leq \alpha_1(L\delta)^2$.
    For $\delta_6$ sufficiently small that 
    \[
    L\delta_6\leq \delta_2,\quad T_7|\log\delta_6|^{-1}\leq \frac{9T_3}{256},
    \]
    applying Lemma~\ref{lem:weak local monotonicity 2} with radius $\frac{3}{16}$ and energy $L\delta$ yields
    \[
    \|A\|_{L^n(B_{3/512}\cap M_t)}\leq KL\delta,\quad\forall\,t\in[0,T_7|\log\delta|^{-1}].
    \]
    Write $g_t=\mathbf x_t^* g_e$ and $\mu_t=\mu_{g_t}$.
    Since $r+D<\frac{3}{512}$, we have the energy estimate
    \begin{equation}\label{ineq:volume ratio by small total curvature-6}
    \left(\int_{E} |A|^n \,d\mu_t\right)^{\frac1n}\leq KL\delta,
    \quad\forall\,t\in[0,t_*].
    \end{equation}

    \noindent\textbf{Step 2. Graphical representation at the terminal time.}

    Set
    \[
    \Omega_*:=C_p(N\cap \mathbf x_{t_*}^{-1}(B_{3r}(p_{t_*}))),\quad Q_s^*:=\mathbf x_{t_*}^{-1}(B_s(p_{t_*}))\cap\Omega_*,\quad s\in(0,2r].
    \]
    Since 
    \[
    |q_{t_*}-p_{t_*}|\geq 8r-2D>6r,\quad \forall\,q\in\partial_M N,
    \]
    $\mathbf x_{t_*}|_N$ is a proper immersion in $B_{6r}(p_{t_*})$.
    Assume that
    \begin{equation}\label{ineq:volume ratio by small total curvature-7}
    2\alpha_3\delta^2\theta^{-1}\leq \varepsilon^2.
    \end{equation}
    By Lemma~\ref{lem:small curvature graph lemma} and \eqref{ineq:volume ratio by small total curvature-5}, $\mathbf x_{t_*}|_{\Omega_*}$ is an embedding and $Q_s^*$ is connected.
    Note that 
    \[
    W_{t_*}\subset B_{2r+2D}(p_{t_*})\subset B_{3r}(p_{t_*}),
    \]
    so $W\subset \mathbf x_{t_*}^{-1}(B_{3r}(p_{t_*}))$.
    Since $W$ is connected and contains $p$, it follows that 
    \[
    W\subset\Omega_*.
    \]
    Then
    \[
    E_{t_*}\subset B_{r+2D}(p_{t_*})\cap W_{t_*}\subset B_{r+2D}(p_{t_*})\cap\mathbf x_{t_*}(\Omega_*),
    \]
    and thus 
    \[
    E\subset Q_{r+2D}^*.
    \]
    For any $q\in Q_{r-2D}^*$, by connectedness there exists a curve 
    \[
    \gamma:[0,1]\to Q_{r-2D}^*,\quad \gamma(0)=p,\quad \gamma(1)=q.
    \]
    Since $p_0=0$, for any $\tau\in[0,1]$, 
    \[
    \begin{aligned}
    \left|\mathbf x_0(\gamma(\tau))\right|
    &\leq |\mathbf x_0(\gamma(\tau))-\mathbf x_{t_*}(\gamma(\tau))|+|\mathbf x_{t_*}(\gamma(\tau))-p_{t_*}|+|p_{t_*}|\\
    &<D+(r-2D)+D=r.
    \end{aligned}
    \]
    Consequently, $\gamma([0,1])\subset \mathbf x_0^{-1}(B_r)$, and thus 
    \[
    q\in C_p(\mathbf x_0^{-1}(B_r))\subset \mathbf x_0^{-1}(B_r)\cap W=E.
    \]
    We conclude that
    \begin{equation}\label{ineq:volume ratio by small total curvature-8}
        Q_{r-2D}^*\subset E\subset Q_{r+2D}^*.
    \end{equation}

    \noindent\textbf{Step 3. Volume comparison and pullback.}
    
    By Lemma~\ref{lem:small curvature graph lemma}, 
    \[
    \begin{aligned}
    \mu_{t_*}(Q_{r+2D}^*)
    &\leq \left[1+C(n,m)\kappa^2(r+2D)^2\right]\omega_n(r+2D)^n\\
    &\leq \left[1+C(n,m)(\kappa^2 r^2+D/r)\right]\omega_n r^n,
    \end{aligned}
    \]
    \[
    \begin{aligned}
    \mu_{t_*}(Q_{r-2D}^*)
    &\geq \left[1-C(n,m)\kappa^2(r-2D)^2\right]\omega_n(r-2D)^n\\
    &\geq \left[1-C(n,m)\left(\kappa^2 r^2+D/r\right)\right]\omega_n r^n.
    \end{aligned}
    \]
    By \eqref{ineq:volume ratio by small total curvature-8} and the estimates \eqref{ineq:volume ratio by small total curvature-4} and \eqref{ineq:volume ratio by small total curvature-5}, we have
    \begin{equation}\label{ineq:volume ratio by small total curvature-9}
        \left|\frac{\mu_{t_*}(E)}{\omega_n r^n}-1\right|
        \leq C(n,m)(\alpha_3\delta^2\theta^{-1}+\sqrt{\alpha_3\theta}\delta)
        \leq C(n,m)(\delta^2\theta^{-1}+\delta\sqrt\theta).
    \end{equation}
    Since
    \[
    \frac{d}{d t}\mu_{t}(E)=-\int_{E}|\mathbf H|^2 \,d\mu_{t}\leq0,
    \]
    by the assumption on volume ratio, 
    \[
    \mu_{t}(E)\leq \mu_{0}(E)\leq \nu r^n.
    \]
    By H\"{o}lder's inequality and \eqref{ineq:volume ratio by small total curvature-6}, 
    \[
    \int_{E}|\mathbf H|^2 \,d\mu_{t}
    \leq n\Bigl(\int_{E}|A|^n\,d\mu_{t}\Bigr)^{\frac2n}\mu_{t}(E)^{1-\frac2n}
    \leq nK^2L^2\delta^2\nu^{1-\frac2n}r^{n-2},
    \]
    and thus 
    \begin{equation}\label{ineq:volume ratio by small total curvature-10}
    \frac{\mu_{0}(E)-\mu_{t_*}(E)}{\omega_n r^n}\leq C(n)K^2L^2\nu^{1-\frac2n}\delta^2\theta\leq C(n,m,\nu)\delta^2\theta.
    \end{equation}
    Combining \eqref{ineq:volume ratio by small total curvature-9} and \eqref{ineq:volume ratio by small total curvature-10} yields 
    \begin{equation}\label{ineq:volume ratio by small total curvature-11}
    -C_1(\delta^2\theta^{-1}+\delta\sqrt\theta)
    \leq \frac{\mu_{t_*}(E)}{\omega_n r^n}-1
    \leq \frac{\mu_{0}(E)}{\omega_n r^n}-1
    \leq C_1(\delta^2\theta^{-1}+\delta\sqrt\theta)+C_2\delta^2\theta,
    \end{equation}
    where $C_1=C_1(n,m)$, $C_2=C_2(n,m,\nu)$.
    Choosing $\theta=\delta^{\frac23}$ and taking $\delta_6=\delta_6(n,m,\nu)$ sufficiently small ensures \eqref{ineq:volume ratio by small total curvature-2}, \eqref{ineq:volume ratio by small total curvature-3} and \eqref{ineq:volume ratio by small total curvature-7}, and gives 
    \[
    \left|\frac{\mu_{0}(E)}{\omega_n r^n}-1\right|
    \leq 2C_1\delta^{\frac43}+C_2\delta^{\frac83}
    \leq 3C_1\delta^{\frac43}.
    \]
\end{proof}

The set $E_{p,r}$ in Theorem~\ref{thm:volume ratio by small total curvature} may not be connected.
The following corollary gives the corresponding estimates for the component containing the base point.

\begin{corollary}\label{cor:almost Euclidean volume ratio of local component}
    Fix $\rho>0$ and $\nu>0$.
    Let $M^n$ be a smooth submanifold properly embedded in $B_\rho\subset\R^{n+m}$.
    Suppose that 
    \[
    \|A\|_{L^n(B_\rho\cap M)}\leq\delta\leq \delta_6
    \]
    and 
    \[
    \sup_{\substack{x\in B_{5\rho/6}\cap M\\0<s\leq\rho/6}} s^{-n}\mu(B_s(x)\cap M)\leq\nu,
    \]
    where $\delta_6=\delta_6(n,m,\nu)$ is from Theorem~\ref{thm:volume ratio by small total curvature}.
    Then, for any $x\in B_{\rho/4}\cap M$ and any $0<r\leq \rho/512$, we have 
    \[
    (1-C\delta^{\frac43})\omega_n r^n
    \leq \mu(C_x(B_r(x)\cap M))
    \leq (1+C\delta^{\frac43})\omega_n r^n.
    \]
    Here $C=C(n,m)$.
\end{corollary}
\begin{proof}
    Fix $x\in B_{\rho/4}\cap M$ and $\rho'=\frac\rho2$.
    Since
    \[
    B_{\rho'}(x)\subset B_{\rho},
    \quad B_{5\rho'/6}(x)\subset B_{5\rho/6},
    \quad \rho'/6\leq\rho/6,
    \]
    the hypotheses of Theorem~\ref{thm:volume ratio by small total curvature} hold on $B_{\rho'}(x)$ with the same $\delta$ and $\nu$.
    Fix $r\in(0,\frac{\rho}{512}]$ and set
    \begin{gather*}
    E_{x,r}:=B_r(x)\cap C_x(B_{2r}(x)\cap M),\quad 
    V_{x,r}:=C_x(B_r(x)\cap M).
    \end{gather*}
    Since $V_{x,r}\subset E_{x,r}$, Theorem~\ref{thm:volume ratio by small total curvature} gives
    \[
    \mu(V_{x,r})\leq \mu(E_{x,r})\leq (1+C\delta^{\frac43})\omega_n r^n.
    \]
    For the lower bound, use the flow and notation from the proof of Theorem~\ref{thm:volume ratio by small total curvature}, centered at \(x\). 
    The curve argument in Step 2 gives
    \[
    \mathbf x_{t_*}(Q_{r-2D}^*)\subset \Phi_{0,t_*}(V_{x,r}) \subset \Phi_{0,t_*}(E_{x,r})\subset\mathbf x_{t_*}(Q_{r+2D}^*).
    \]
    Hence, by the evolution of the volume element and the terminal volume estimate in Step 3,
    \[
    \mu_{0}(V_{x,r})
    \geq \mu_{t_*}(\Phi_{0,t_*}(V_{x,r}))
    \geq \mu_{t_*}(\mathbf x_{t_*}(Q_{r-2D}^*))
    \geq (1-C\delta^{\frac43})\omega_n r^n.
    \]
    This completes the proof.
\end{proof}

\section{Regularity of tangent flows}\label{sec6}

In this section, we restrict to the hypersurface case $m=1$ and prove the main result, namely the regularity of the support of tangent flows. 
The local volume estimate from Section~\ref{sec4}, together with compactness for curvature varifolds, is used to analyze the limit away from curvature concentration points.
We then treat the remaining isolated singularities using uniqueness of tangent planes and graphical representation from Appendix~\ref{appB}.

\begin{theorem}[Regularity of tangent flow supports]\label{thm:tangent-flow-support-regularity}
Fix $R>0$ and $0<T<\infty$.
Let $\{M_t^n\}_{t\in[0,T)}$ be a smooth mean curvature flow of
properly embedded hypersurfaces in $\mathbb{R}^{n+1}$. 
Suppose that there exist constants $D,E<\infty$ such that
\begin{equation}
\label{eq:tangent-flow-support-regularity 0-1}
    \sup_{t\in[0,T)}
    \sup_{x\in\R^{n+1}}
    \sup_{0<r\leq R}
    \frac{\mu_t(B_r(x)\cap M_t)}{\omega_n r^n}
    \leq D,
\end{equation}
and
\begin{equation}
\label{eq:tangent-flow-support-regularity 0-2}
    \sup_{t\in[0,T)}
    \int_{B_R\cap M_t}|A|^n\,d\mu_t
    \leq E.
\end{equation}
Then, for every tangent flow $\{\nu_\tau\}_{\tau<0}$ at a singular point $(x_0,T)$ with $x_0\in B_R$, $\spt\nu_{-1}$ is a smooth embedded self-shrinker.
\end{theorem}

We first establish a one-sidedness lemma for limits of boundaries.
It will be used to show that distinct limiting sheets in the proof of Theorem~\ref{thm:tangent-flow-support-regularity} cannot cross.
\begin{lemma}\label{lem:one-sidedness of limit}
    Let $U\subset \R^{n+1}$ be a connected open set, and let $\Gamma$ be a nonempty smooth hypersurface properly embedded in $U$.
    Write
    \[
    U\setminus\Gamma=U^+\sqcup U^-,\qquad \partial_U U^+=\partial_U U^-=\Gamma,
    \]
    where $U^\pm$ are open.
    Let $E_j$ be a sequence of open sets with smooth boundary in $U$, and set $D_j=\partial E_j\cap U$.
    Suppose that, for some measurable set $E\subset U$,
    \[
    \chi_{E_j}\to\chi_E
    \quad\text{in }L^1_{\mathrm{loc}}(U),
    \qquad
    V_{D_j}\rightharpoonup V_\Gamma
    \quad
    \text{as varifolds in }U,
    \]
    where the varifolds have multiplicity one. 
    Assume also that
    \[
    \sup_j\|A\|_{L^p(D_j)}<\infty
    \]
    for some $p>1$.
    %Then $E$ agrees almost everywhere in $U$ with either $U^+$ or $U^-$.
    Then 
    \[
    E=U^+\quad \text{or}\quad E=U^-\quad\text{a.e. in }U.
    \]
\end{lemma}
\begin{proof}
    %Let $V_j$ and $V_\Gamma$ be the multiplicity-one varifolds induced by $D_j$ and $\Gamma$, respectively.
    Choose a cylindrical neighborhood $Q\Subset U$ such that $\Gamma\cap Q$ is a connected graph, and extend the outward unit normal of $\partial U^+$ to a smooth unit vector field $N$ on a neighborhood of $\overline{Q}$.
    Let $N_j$ be the outward unit normal of $\partial E_j\cap U$ and define 
    \[
    f_j:=\langle N_j,N\rangle\quad\text{on }D_j\cap Q.
    \]
    By the varifold convergence $V_{D_j}\rightharpoonup V_\Gamma$,
    \begin{equation}\label{eq:one-sidedness of limit-1}
    \int_{D_j\cap Q}(1-f_j^2)\,d\mathcal H^n\to0,
    \end{equation}
    and thus
    \[
    \mathcal H^n(T_j)\to0,\qquad T_j:=\{\,x\in D_j\cap Q\mid |f_j(x)|\leq \tfrac12\,\}.
    \]
    By the uniform $L^p$ bound and H\"{o}lder's inequality, 
    \[
    \begin{aligned}
    \int_{T_j}|\nabla^{D_j}f_j|\,d\mathcal H^n
    &\leq \int_{T_j}|A_j|\,d\mathcal H^n+\sup_Q|\nabla N|\mathcal H^n(T_j)\\
    &\leq K\cdot\mathcal H^n(T_j)^{\frac1q}+\sup_Q|\nabla N|\cdot\mathcal H^n(T_j)\to0,
    \end{aligned}
    \]
    where 
    \[
    K:=\sup_j\|A\|_{L^p(D_j)}<\infty,\qquad q=\frac{p}{p-1}.
    \]
    By the co-area formula, there exist regular values 
    \(
    a_j\in(-\frac12,\frac12)
    \)
    such that 
    \[
    \mathcal H^{n-1}(L_j)\to 0,\qquad L_j:=\{f_j=a_j\}\cap D_j\cap Q.
    \]
    Define 
    \[
    D_j^+:=\{f_j>a_j\}\cap D_j\cap Q,\qquad D_j^-:=\{f_j<a_j\}\cap D_j\cap Q,
    \]
    whose closures in $D_j\cap Q$ are smooth manifolds with boundary $L_j$.
    %Let $V_j^\pm$ be the varifolds induced by $D_j^\pm$, respectively.
    For any $\Omega\Subset Q$, by the first variation formula and H\"{o}lder's inequality, 
    \[
    \begin{aligned}
    \|\delta V_{D_j^\pm}\|(\Omega)
    &\leq \int_{\Omega\cap D_j^\pm}|\mathbf H_j|\,d\mathcal H^n+\mathcal H^{n-1}(\Omega\cap L_j)\\
    &\leq \sqrt n K\cdot\|V_{D_j^\pm}\|(\Omega)^{\frac1q}+\mathcal H^{n-1}(L_j).
    \end{aligned}
    \]
    Since $V_{D_j}\rightharpoonup V_\Gamma$, 
    \[
    \sup_j\| V_{D_j^\pm}\|(\Omega)\leq \sup_j\| V_{D_j}\|(\Omega)<\infty.
    \]
    Applying Allard's compactness theorem then yields 
    \[
    V_{D_j^+}\rightharpoonup W^+,\qquad 
    V_{D_j^-}\rightharpoonup W^-\quad
    \text{as integral varifolds in }Q.
    \]
    Since $V_{D_j^+}+V_{D_j^-}=V_{D_j}$, we obtain
    \[
    W^++W^-=V_{\Gamma\cap Q}.
    \]
    Write
    \[
    W^\pm=\theta_\pm V_{\Gamma\cap Q},
    \]
    where
    \[
    \theta_\pm\in\{0,1\},\quad
    \theta_++\theta_-\equiv1,\quad 
    \text{a.e. on $\Gamma\cap Q$}.
    \]
    For any $X\in C_c^1(Q;\R^{n+1})$ with $X(x)\in T_x\Gamma$ ($x\in\Gamma\cap Q$), 
    \[
    \delta V_{D_j^\pm}(X)=
    -\int_{D_j^\pm}\mathbf H_j\cdot X\,d\mathcal H^n
    +\int_{L_j}X\cdot\mathbf\eta_j^\pm\,d\mathcal H^{n-1},
    \]
    where $\eta_j^\pm$ are outward unit conormal vector fields of $\partial D_j^\pm$.
    Note that 
    \[
    \begin{aligned}
    \left|\int_{D_j^\pm}\mathbf H_j\cdot X\,d\mathcal H^n\right|
    &=\left|\int_{D_j^\pm}\mathbf H_j\cdot P_{T_x D_j^\perp}X\,d\mathcal H^n\right|\\
    &\leq \sqrt n K\|P_{T_x D_j^\perp}X\|_{L^q(D_j^\pm)}\\
    &\to \sqrt n K \|P_{T_x\Gamma^\perp}X\|_{L^q(W^\pm)}
    =0
    \end{aligned}
    \]
    and 
    \[
    \left|\int_{L_j}X\cdot\eta_j^\pm\,d\mathcal H^{n-1}\right|
    \leq \|X\|_{L^\infty(Q)}\mathcal H^{n-1}(L_j)
    \to0.
    \]
    Hence,
    \[
    \int_{\Gamma\cap Q} \theta_\pm \Div_\Gamma Y \,d\mathcal H^n=0,
    \quad \forall\,Y\in C_c^1(\Gamma\cap Q;T\Gamma),
    \]
    which together with the connectedness of $\Gamma\cap Q$ implies that $\theta_\pm$ are constant almost everywhere on $\Gamma\cap Q$.
    Then there are only two possibilities: 
    \begin{equation}\label{eq:one-sidedness of limit-2}
    (\theta_+,\theta_-)\equiv(1,0)\text{ or } (0,1).
    \end{equation}
    Since $\partial E_j\cap Q=D_j\cap Q$, it follows by the divergence theorem that, for any $X\in C_c^1(Q;\R^{n+1})$, 
    \[
    \begin{aligned}
    \int_{E_j}\Div X\,d\mathcal L^{n+1}
    &= \int_{D_j^+}X\cdot N\,d\mathcal H^n-\int_{D_j^-}X\cdot N\,d\mathcal H^n+\mathcal E_j(X),
    \end{aligned}
    \]
    where
    \[
    \begin{aligned}
    \mathcal E_j(X)
    &:=\int_{D_j^+}X\cdot (N_j-N)\,d\mathcal H^n+\int_{D_j^-}X\cdot (N_j+N)\,d\mathcal H^n.
    \end{aligned}
    \]
    Since 
    \[
    |N_j-N|^2\leq 4(1-f_j^2)\quad\text{on }D_j^+,\qquad
    |N_j+N|^2\leq 4(1-f_j^2)\quad\text{on }D_j^-,
    \]
    by H\"{o}lder's inequality and \eqref{eq:one-sidedness of limit-1} we have
    \[
    \begin{aligned}
        |\mathcal E_j(X)|
        &\leq 4\|X\|_{L^\infty(Q)}\cdot\|V_{D_j}\|(\spt X)^{\frac12}\cdot\left(\int_{D_j\cap\spt X}(1-f_j^2)\,d\mathcal H^n\right)^{\frac12}
        \to0.
    \end{aligned}
    \]
    Hence, by \eqref{eq:one-sidedness of limit-2}, 
    \[
    \begin{aligned}
    \int_E\Div X\,d\mathcal L^{n+1}
    &=\int_Q X\cdot N\,d\|W^+\|-\int_Q X\cdot N\,d\|W^-\|\\
    &=\int_Q (\theta_+-\theta_-)X\cdot N\,d\|V_\Gamma\|\\
    &=\theta \int_{\Gamma\cap Q} X\cdot N\,d\mathcal H^n,
    \end{aligned}
    \]
    where $\theta:=\theta_+-\theta_-\in\{\pm1\}$ is a constant.
    By the definition of $U^\pm$, 
    \[
    \int_{U^\pm}\Div X\,d\mathcal L^{n+1}=\pm\int_{\Gamma\cap Q} X\cdot N\,d\mathcal H^n,
    \]
    which implies 
    \[
    \chi_E=\chi_{U^+}\quad\text{or}\quad\chi_E=\chi_{U^-}\quad\text{a.e. in }Q.
    \]
    In addition, for any $X\in C_c^1(U\setminus\Gamma;\R^{n+1})$, 
    \[
    \int_E\Div X\,d\mathcal L^{n+1}=\lim_j\int_{D_j}X\cdot N_j\,d\mathcal H^n=0,
    \]
    which implies $\chi_E$ is locally constant in $U\setminus\Gamma$. 
    By the connectedness of $U$, we obtain
    \[
    \chi_E=\chi_{U^+}\quad\text{or}\quad\chi_E=\chi_{U^-}\quad\text{a.e. in }U.
    \]
\end{proof}

\begin{proof}[Proof of Theorem~\ref{thm:tangent-flow-support-regularity}]
Choose a sequence $\lambda_j\downarrow0$ such that the
parabolically rescaled flows
\[
M_\tau^{\lambda_j}
:=\lambda_j^{-1}\left(M_{T+\lambda_j^2\tau}-x_0\right),
\quad\tau\in[-\lambda_j^{-2}T,0),
\]
with associated area measures $\mu_\tau^{\lambda_j}$,
converge to $\{\nu_\tau\}_{\tau<0}$ as Brakke flows.
Set
\[
\Sigma:=\operatorname{spt}\nu_{-1}.
\]
%We also use $\nu_{-1}$ to denote the associated integral varifold when no confusion arises.
We prove that $\Sigma$ is smooth and embedded, following the approach of Ilmanen~\cite{ilmanen2026singularitiesmeancurvatureflow}.

\noindent\textbf{Step 1. Choosing a refined sequence.}

The local volume bound \eqref{eq:tangent-flow-support-regularity 0-1} implies polynomial growth, and hence  
\[
\Phi(s):=\int_{M_s}\varrho_{x_0,T}(y,s)\,d\mu_s(y)\leq C(n,D,T,R)<\infty,
\quad\forall\,s\in[0,T).
\]
For sufficiently large $j$ and $\tau\in(-2,-1/2)$, set
\[
\mathcal E_j(\tau)
:=
\int_{M_\tau^{\lambda_j}}\varrho_{0,0}(y,\tau)
\left|\mathbf H_{M_\tau^{\lambda_j}}+\frac{y^{\perp_j}}{-2\tau}\right|^2
\,d\mu_\tau^{\lambda_j}(y),
\]
where $\perp_j$ denotes the orthogonal projection onto $N_y M_\tau^{\lambda_j}$. By Huisken's monotonicity formula,
\[
\begin{aligned}
\int_{-2}^{-\frac12}\mathcal E_j(\tau)\,d\tau
\leq
\Phi(T-2\lambda_j^2)-\Phi\left(T-\lambda_j^2/2\right) 
\to0
\quad\text{as }j\to\infty.
\end{aligned}
\]
After passing to a subsequence, we have $\mathcal E_j(\tau)\to0$ for almost every $\tau\in(-2,-1/2)$.
Since $\mu_\tau^{\lambda_j}\rightharpoonup\nu_\tau$ for every $\tau<0$, and $\nu_s\rightharpoonup\nu_{-1}$ as $s\to-1$ by self-similarity,
a diagonal argument yields times $\tau_j\to-1$ such that
\begin{equation}\label{eq:tangent-flow-support-regularity 1-1}
\mu_{\tau_j}^{\lambda_j}\rightharpoonup\nu_{-1},
\qquad
\mathcal{E}_j(\tau_j)\to0.
\end{equation}
Set
\[
    \mu_j:=\mu_{\tau_j}^{\lambda_j},\qquad
    \Sigma_j:=M_{\tau_j}^{\lambda_j}.
\]
For every fixed $L>0$, $B_{\lambda_j L}(x_0)\subset B_R$ for sufficiently large $j$.
By \eqref{eq:tangent-flow-support-regularity 0-1}--\eqref{eq:tangent-flow-support-regularity 0-2} and scale invariance, for sufficiently large $j$, the rescaled time-slices therefore satisfy
\[
\int_{B_L\cap \Sigma_j}|A_{\Sigma_j}|^n\,d\mu_j\leq E,
\qquad
\sup_{y\in\R^{n+1}}\sup_{0<r\leq R/{\lambda_j}}\frac{\mu_j(B_r(y)\cap\Sigma_j)}{\omega_n r^n}\leq D.
\]
Set 
\begin{equation*}
    \sigma_j
    :=
    |A_{\Sigma_j}|^n
    \mathcal{H}^n\llcorner\Sigma_j.
\end{equation*}
After passing to a subsequence,
\(
    \sigma_j\rightharpoonup\sigma
\)
locally as Radon measures. 
For every $L>0$,
\[
\sigma(B_L)\leq \liminf_{j\to\infty}\sigma_j(B_L)\leq E.
\]
Letting $L\to\infty$ gives 
\(
\sigma(\R^{n+1})\leq E
\).
Define
\[
    \mathcal S=\left\{\,q\in \Sigma\mid\sigma(\{q\})\geq\varepsilon_0\,\right\}.
\]
Here, $\varepsilon_0$ is chosen sufficiently small that
\[
\varepsilon_0^{\frac1n}\leq \delta_6,\qquad C\varepsilon_0^{\frac{4}{3n}}\leq\lambda_0,
\]
where 
$\lambda_0=\lambda_0(n,1,2n)\in(0,1/16]$ is from Lemma~\ref{lem:Allard regularity}, 
$\delta_6=\delta_6(n,1,\omega_n D)$ and $C=C(n,1)$ are from Corollary~\ref{cor:almost Euclidean volume ratio of local component}.
Note that $\mathcal S$ is finite, with 
\[
\#\mathcal S\le E/\varepsilon_0.
\]

\noindent\textbf{Step 2. Regularity away from $\mathcal S$.}

Fix $p\in \Sigma\setminus \mathcal S$.
Let $r\in(0,1]$ be any radius such that, with $s=2^{-9}r$, 
\begin{equation}\label{eq:tangent-flow-support-regularity 2-1}
     \sigma(\overline B_r(p))<\varepsilon_0,
     \qquad \tfrac12\max\{1,(D\omega_n)^{\frac1{2n}}\}s(|p|+s)\leq\lambda_0.
\end{equation}
Such radii exist since $\sigma(\{p\})<\varepsilon_0$.
The weak convergence gives
\begin{equation*}
    \int_{B_r(p)\cap\Sigma_j}
    |A_{\Sigma_j}|^n\,d\mu_j
    <\varepsilon_0,
\end{equation*}
for sufficiently large $j$.
Let
\[
D_{j,1},\dots, D_{j,N_j}
\]
be the connected components of $B_{2s}(p)\cap\Sigma_j$ that meet $B_s(p)$.
For each integer $l\in\{1,\dots, N_j\}$ and any $q_{j,l}\in B_s(p)\cap D_{j,l}$, since $B_{s/2}(q_{j,l})\subset B_{2s}(p)$, applying Corollary~\ref{cor:almost Euclidean volume ratio of local component} yields
\[
\begin{aligned}
\mu_j(D_{j,l})
&\geq \mu_j\left(C_{q_{j,l}}(B_{s/2}(q_{j,l})\cap\Sigma_j)\right)
\geq (1-C\varepsilon_0^{\frac{4}{3n}})\omega_n(s/2)^n.
\end{aligned}
\]
Then 
\[
N_j (1-\lambda_0)\omega_n(s/2)^n\leq \mu_j(B_{2s}(p)\cap\Sigma_j)\leq D\omega_n(2s)^n, 
\]
and thus 
\[
N_j\leq C(n)D.
\]
After passing to a subsequence, we may assume that
\[
N_j=N_p,\quad \forall\,j.
\]
For each integer $l\in\{1,\dots,N_p\}$, let $V_{j,l}$ be the multiplicity-one integral varifold in $B_{2s}(p)$ induced by $D_{j,l}$.
For any open set $W\Subset B_{2s}(p)$, 
\begin{gather*}
\|V_{j,l}\|(W)\leq D\omega_n(2s)^n,\\
\|\delta V_{j,l}\|(W)
\leq \sqrt n\varepsilon_0^{\frac1n}\|V_{j,l}\|(W)^{1-\frac1n}
\leq C(n)\varepsilon_0^{\frac1n}D^{1-\frac1n}s^{n-1}.
\end{gather*}
By Allard's compactness theorem, after taking a subsequence,
\[
V_{j,l}\rightharpoonup V_l\quad \text{as varifolds in }B_{2s}(p).
\]
By the definition of $D_{j,l}$, 
\[
\Sigma_j\cap B_s(p)=\bigsqcup_{l=1}^{N_p} (D_{j,l}\cap B_s(p)),
\]
which, by passing to the limit, gives the decomposition
\[
\Sigma\cap B_s(p)=\bigcup_{l=1}^{N_p}\left(\spt\|V_l\|\cap B_s(p)\right),\qquad
\nu_{-1}\llcorner B_s(p)=\sum_{l=1}^{N_p} \|V_l\|\llcorner B_s(p).
\]
Set
\[
I_p:=\bigl\{\,l\in\{1,\dots,N_p\}\mid p\in\spt\|V_l\|\,\bigr\}.
\]
By the decomposition, $I_p\neq\emptyset$.
For every $l\in I_p$, choose $q_{j,l}\in D_{j,l}$ such that
\[
q_{j,l}\to p,\qquad d_{j,l}:=|q_{j,l}-p|\to0.
\]
Since
\[
D_{j,l}\subset C_{q_{j,l}}(B_{2(s+d_{j,l})}(q_{j,l})\cap \Sigma_j),\qquad
B_s(p)\subset B_{s+d_{j,l}}(q_{j,l}),
\]
applying Theorem~\ref{thm:volume ratio by small total curvature} then yields
\[
\begin{aligned}
\mu_j(B_s(p)\cap D_{j,l})
&\leq \mu_j\left(B_{s+d_{j,l}}(q_{j,l})\cap C_{q_{j,l}}(B_{2(s+d_{j,l})}(q_{j,l})\cap \Sigma_j)\right)\\
&\leq (1+C(n)\varepsilon_0^{\frac{4}{3n}})\omega_n(s+d_{j,l})^n
\end{aligned}
\]
and thus 
\begin{equation}\label{eq:tangent-flow-support-regularity 2-2}
\|V_l\|(B_s(p))\leq (1+\lambda_0)\omega_n s^n,\quad \forall\,l\in I_p.
\end{equation}
For any $X\in C_c^1(B_s(p);\R^{n+1})$, by \eqref{eq:tangent-flow-support-regularity 1-1},
\[
\left|\int_{D_{j,l}}\left(\mathbf H_{\Sigma_j}+\frac{x^{\perp_j}}{-2\tau_j}\right)\cdot X\,d\mu_j\right|
\leq \mathcal E_j(\tau_j)^{\frac12} \left(\int_{D_{j,l}}\varrho_{0,0}^{-1}|X|^2\,d\mu_j\right)^{\frac12}
\to0,
\]
which implies
\begin{equation}\label{eq:tangent-flow-support-regularity 2-3}
\mathbf H_{V_l}+\frac{x^{\perp_{V_l}}}{2}=0\quad \|V_l\|\text{-a.e. in } B_s(p).
\end{equation}
By \eqref{eq:tangent-flow-support-regularity 2-3} and \eqref{eq:tangent-flow-support-regularity 2-1},
\[
|\mathbf H_{V_l}|\leq \frac{|p|+s}{2}\quad
\|V_l\|\text{-a.e. in } B_s(p),
\]
and then
\begin{equation}\label{eq:tangent-flow-support-regularity 2-4}
s^{\frac12}\|\mathbf H_{V_l}\|_{L^{2n}(B_s(p))}
\leq \tfrac12s(|p|+s)(D\omega_n)^{\frac{1}{2n}}
\leq \lambda_0.
\end{equation}
Combining \eqref{eq:tangent-flow-support-regularity 2-2} and \eqref{eq:tangent-flow-support-regularity 2-4}, by Allard's regularity theorem we see that for each $l\in I_p$, 
\[
\spt\|V_l\|\cap B_{\gamma s}(p)
=\graph_{P_l} u_l\cap B_{\gamma s}(p),
\]
with 
\begin{equation}\label{eq:tangent-flow-support-regularity 2-5}
s^{-1}\sup_{B_{\gamma s}^{P_l}(p)}|u_l|+\sup_{B_{\gamma s}^{P_l}(p)}|Du_l|\leq C(n),
\end{equation}
where $\gamma=\gamma(n,1,2n)\in(0,1)$, $P_l$ is an $n$-hyperplane passing through $p$, and $u_l\in C^{1,\frac12}(B_{\gamma s}^{P_l}(p))$.
Elliptic bootstrapping for the self-shrinker equation \eqref{eq:tangent-flow-support-regularity 2-3} gives $u_l\in C^\infty$.
For $y\in \spt\|V_l\|\cap B_{s}(p)$, since $B_{\rho_y}(y)\subset B_s(p)$ with $\rho_y=s-|y-p|$, by the monotonicity formula for varifolds and \eqref{eq:tangent-flow-support-regularity 2-1}--\eqref{eq:tangent-flow-support-regularity 2-2} we have
\[
\Theta^n(\|V_l\|,y)
\leq e^{\frac{|p|+s}{2}\rho_y}\frac{\|V_l\|(B_{\rho_y}(y))}{\omega_n\rho_y^n}
\leq e^{\lambda_0}(1+\lambda_0)\biggl(1-\frac{|y-p|}{s}\biggr)^{-n}
<2,
\]
provided that $|y-p|$ is sufficiently small.
By integrality, each $V_l$ has multiplicity one around $p$.
Since $N_p<\infty$ and $p\notin\spt\|V_l\|$ for $l\notin I_p$, we may choose a cylindrical neighborhood $Q_p\Subset B_{\gamma s}(p)$ about $p$ and a common $n$-plane $P_p$ through $p$ such that
\[
\nu_{-1}\llcorner Q_p=\sum_{l\in I_p}\|V_{\Gamma_l}\|,\qquad
\Gamma_l:=\graph_{P_p}\tilde u_l,
\]
where $\tilde u_l\in C^\infty(P_p\cap Q_p)$.
Note that each $D_{j,l}$ separates $B_{2s}(p)$ into two components.
For $a,b\in I_p$ with $a\neq b$, choose the side of $D_{j,a}$ not containing $D_{j,b}$ and the side of $D_{j,b}$ not containing $D_{j,a}$.
Restricting these sides to $Q_p$ gives open sets $E_{j,a},E_{j,b}$ such that
\[
\partial_{Q_p}E_{j,a}=D_{j,a}\cap Q_p,\qquad 
\partial_{Q_p} E_{j,b}=D_{j,b}\cap Q_p,\qquad
E_{j,a}\cap E_{j,b}=\emptyset.
\]
By BV compactness,
\[
\chi_{E_{j,a}}\to\chi_{E_a},\qquad \chi_{E_{j,b}}\to\chi_{E_b}\quad \text{in } L_{\mathrm{loc}}^1(Q_p),
\]
and thus 
\[
\chi_{E_a}\chi_{E_b}=0\quad\text{a.e. in } Q_p.
\]
Note that
\[
V_{j,l}\llcorner Q_p\rightharpoonup V_{\Gamma_l}\quad\text{as varifolds in }Q_p,\quad l\in\{a,b\}.
\]
By Lemma~\ref{lem:one-sidedness of limit}, $E_l$, for $l\in\{a,b\}$, coincides with one side of $\Gamma_l$ up to null sets.
Since $E_a$ and $E_b$ are disjoint, after possibly relabeling, $E_a$ lies above $\Gamma_a$ and $E_b$ lies below $\Gamma_b$. 
By continuity we have
\[
\tilde u_a\geq \tilde u_b\quad\text{in }P_p\cap Q_p.
\]
By the self-shrinker equation \eqref{eq:tangent-flow-support-regularity 2-3} and the strong maximum principle, 
\[
\tilde u_a\equiv \tilde u_b\quad\text{in }P_p\cap Q_p.
\]
As a result, 
\[
\nu_{-1}\llcorner Q_p=\#I_p\mathcal H^n\llcorner\Gamma_p,\qquad\Gamma_p:=\graph_{P_p}\tilde u_p,
\]
for some $\tilde u_p\in C^\infty(P_p\cap Q_p)$.
Thus $\Sigma\setminus \mathcal S$ is smooth and embedded.
In addition, since $\sigma_j=|A_{\Sigma_j}|^n\mathcal H^n\llcorner\Sigma_j\rightharpoonup\sigma$, by the lower-semicontinuity from Lemma~\ref{lem:compactness and lower-semicontinuity of curvature varifolds} we have
\begin{equation}\label{eq:tangent-flow-support-regularity 2-6}
    |A_\Sigma|^n\nu_{-1}\leq\sigma\quad\text{on }\R^{n+1}\setminus\mathcal S.
\end{equation}

\noindent\textbf{Step 3. Removal of isolated singularities.}

Fix $q\in \mathcal S$.
Take $\bar r\in(0,1]$ sufficiently small that 
\begin{gather*}
    B_{2\bar r}(q)\cap\mathcal S=\{q\},\qquad
    \sigma(B_{2\bar r}(q)\setminus\{q\})< \varepsilon_0,\qquad
    \max\{1,(D\omega_n)^{\frac{1}{2n}}\}\bar r(|q|+\bar r)\leq\lambda_0.
\end{gather*}
For $x\in \Sigma\cap (B_{\bar r}(q)\setminus\{q\})$ with $d:=|x-q|$, since $\overline{B}_{d/2}(x)\subset B_{2d}(q)\setminus\{q\}$, \eqref{eq:tangent-flow-support-regularity 2-1} holds for $B_{d/2}(x)$.
By Step 2, $\Sigma$ is smooth and embedded near $x$.
There exists a self-shrinking graphical sheet through $x$ of scale $\rho_x=2^{-10}\gamma d$, agreeing locally with $\Sigma$ near $x$ and satisfying \eqref{eq:tangent-flow-support-regularity 2-5} and 
\[
\rho_x|x|\leq 2^{-10}\gamma \bar r(|q|+\bar r)\leq 1.
\] 
Applying Proposition~\ref{prop:interior estimate for self-shrinker} to this sheet and using the arbitrariness of $x$ then yields 
\begin{equation}\label{eq:tangent-flow-support-regularity 3-1}
   \sup_{x\in\Sigma\cap(B_{\bar r}(q)\setminus\{q\})}|x-q|^{k+1}|\nabla^k A_\Sigma|(x)\leq C_k,\quad C_k=C_k(n,k),\quad\forall\, k\geq0.
\end{equation}
For any sequence $\{x_i\}\subset\Sigma\setminus\{q\}$ with $x_i\to q$, consider the rescaling 
\[\Sigma^{(i)}:=d_i^{-1}(\Sigma-x_i),\qquad d_i:=|x_i-q|\to0.\]
Since 
\(B_{d_i/2}(x_i)\subset A_{d_i/2,2d_i}(q)\),
by scaling and \eqref{eq:tangent-flow-support-regularity 3-1} we have
\[
\sup_{\Sigma^{(i)}\cap B_{1/2}} |\nabla^kA_{\Sigma^{(i)}}|
\leq 2^{k+1}C_k,\quad
\forall\,k\geq0.
\]
For any subsequence $\{x_{i_j}\}$, after passing to a further subsequence, we have the smooth convergence 
\[
\Sigma^{(i_j)}\cap B_{1/2} \to \Sigma^{\infty}\cap B_{1/2},
\]
where $\Sigma^\infty$ is a minimal hypersurface in $B_{1/2}$.
On the other hand, by \eqref{eq:tangent-flow-support-regularity 2-6},
\[
\int_{\Sigma\cap A_{d_i/2,2d_i}(q)}|A_\Sigma|^n\,d\mathcal H^n
\leq\sigma(B_{2d_i}(q)\setminus\{q\})
\to0,
\]
and then
\[
\|A_{\Sigma^{(i)}}\|_{L^n(B_{1/2})}
=\|A_\Sigma\|_{L^n(B_{d_i/2}(x_i))}
\leq \|A_\Sigma\|_{L^n(A_{d_i/2,2d_i}(q))}\to0.
\]
Hence, 
\[
\|A_{\Sigma^\infty}\|_{L^n(B_{1/2})}=0,
\]
which implies 
\[
d_{i_j}^{k+1}|\nabla^k A_\Sigma|(x_{i_j})=|\nabla^k A_{\Sigma^{(i_j)}}|(0)\to0,\quad\forall\,k\geq0.
\]
By the arbitrariness of $\{x_{i_j}\}$ and $\{x_i\}$, we conclude that 
\begin{equation}\label{eq:tangent-flow-support-regularity 3-2}
\lim_{\substack{x\to q\\x\in\Sigma\setminus\{q\}}}|x-q|^{k+1}|\nabla^k A_\Sigma|(x)=0,\quad\forall\,k\geq0.
\end{equation}
By the self-shrinker equation \eqref{eq:tangent-flow-support-regularity 2-3}, $|\mathbf H_\Sigma|$ is uniformly bounded around $q$.
By the monotonicity formula for varifolds and the integrality, we have
\[
1\leq \Theta^n(\nu_{-1},q)\leq D.
\]
%So, by embeddedness, the support of the tangent cone at $q$ is a hyperplane. We next show the tangent cone is unique.  
By Lemma~\ref{lem:L-S ineq}, there exist $r_q\in (0,\bar r)$, a common $n$-plane $P_q$ through $q$ and smooth functions $u_1\leq u_2\leq\cdots\leq u_{\theta_q}$ on $B_{r_q}^{P_q}(q)\setminus\{q\}$ such that
\begin{equation*}
  \Sigma\cap\left(B_{r_q}(q)\setminus\{q\}\right)
=\left(\bigcup_{\alpha=1}^{\theta_q}\graph_{P_q}u_\alpha\right)\cap \left(B_{r_q}(q)\setminus\{q\}\right),
\end{equation*}
where $\theta_q:=\Theta^n(\nu_{-1},q)$.
For each $\alpha\in\{1,\dots,\theta_q\}$, the graphical function $u_\alpha$ satisfies
\[
u_\alpha(z)=o(|z-q|),\qquad 
|D u_\alpha(z)|=o(1)\quad
\text{as }z\to q.
\]
Consequently, each $u_\alpha$ extends across $q$ as a $C^1$ function satisfying
\[
u_\alpha(q)=0,
\qquad
Du_\alpha(q)=0.
\]
For sufficiently small $\epsilon>0$, $u=u_\alpha$ satisfies the graphical equation
\[
\Div F+G=0,\qquad
F:=\frac{Du}{\sqrt{1+|Du|^2}},\quad
G:=\frac12\langle z+u\eta_q,N_u\rangle,
\]
in the weak sense in $A_{\epsilon,r_q}^{P_q}(q)$, where $\eta_q$ is the unit normal of $P_q$ and $N_u$ is the associated upward unit normal to $\graph_{P_q}u$, with 
\[
\sup_{B_{r_q}^{P_q}(q)}|F|\leq1,\qquad \sup_{B_\epsilon^{P_q}(q)}|G|\leq C_q.
\]
For any $\phi\in C_c^1(B_{r_q}^{P_q}(q))$, 
\[
0
%=\int_{A_{\epsilon,r_q}^{P_q}(q)}(\Div F+G)\phi\,d\mathcal H^n
=\int_{A_{\epsilon,r_q}^{P_q}(q)}(-F\cdot D\phi+G\phi)\,d\mathcal H^n+\int_{\partial B_\epsilon^{P_q}(q)}\phi F\cdot\xi_\epsilon\,d\mathcal H^{n-1},
\]
where $\xi_\epsilon$ is the inward unit normal to $B_\epsilon^{P_q}(q)$, and thus 
\[
\begin{aligned}
\left|\int_{B_{r_q}^{P_q}(q)}(-F\cdot D\phi+G\phi)\,d\mathcal H^n\right|
&\leq (\|D\phi\|_{L^\infty}+C_q\|\phi\|_{L^\infty})\omega_n\epsilon^n+\|\phi\|_{L^\infty} n\omega_n\epsilon^{n-1}\\
&\to0\quad\text{as }\epsilon\to0.
\end{aligned}
\]
Thus each $u_\alpha$ satisfies the same equation across $q$ in the weak sense in $B_{r_q}^{P_q}(q)$. 
Elliptic regularity implies that each $u_\alpha$ is smooth across $q$.
By the strong maximum principle, 
\[
u_1\equiv u_2\equiv\cdots\equiv u_{\theta_q}\quad \text{in } B_{r_q}^{P_q}(q).
\]
Hence,
\[
\Sigma\cap B_{r_q}(q)=\Gamma_q,\qquad
\Gamma_q:=\graph_{P_q}u_q\cap B_{r_q}(q),
\]
for some $u_q\in C^\infty(B_{r_q}^{P_q}(q))$.
The self-shrinker equation implies that the multiplicity is constant near $q$, which equals $\Theta^n(\nu_{-1},q)=\theta_q$.
As a result,
\[
\nu_{-1}\llcorner B_{r_q}(q)= \theta_q \mathcal H^n\llcorner\Gamma_q.
\]
Therefore the singularity $q$ is removable. 
Since $\mathcal S$ is finite, this proves that $\Sigma$ is a smooth embedded self-shrinker.
\end{proof}

\begin{proof}[Proof of Theorem~\ref{thm:intro-tangent-flow}]
For any singular point $(x_0,T)$, choose $R>|x_0|$.
Since $M_0$ is smooth and closed, Huisken's monotonicity formula yields a constant $D=D(M_0)<\infty$ such that
\[
\mu_t(B_r(x)\cap M_t)\leq D\omega_n r^n
\quad\text{for all }t\in[0,T),\ x\in\R^{n+1},\ r>0.
\]
The assumption \eqref{eq0.2} gives the local total-curvature bound \eqref{eq:tangent-flow-support-regularity 0-2} required in Theorem~\ref{thm:tangent-flow-support-regularity}.
Apply that theorem to obtain the conclusion.
\end{proof}

\appendix

\section{Uniqueness of the tangent plane}\label{appB}
In this appendix, we prove uniqueness of the tangent plane and a
graphical representation near an isolated singularity of a self-shrinker
under the curvature decay assumption \eqref{eq:curvature-smallness}.

\begin{lemma}\label{lem:L-S ineq}
Let $n\geq2$, and let $\nu$ be an integral $n$-varifold in
$B_{r_0}(q)\subset\mathbb R^{n+1}$.
Suppose that $\nu$ satisfies the self-shrinker equation
\[
    \mathbf H_\nu(x)+\tfrac12x^{\perp_\nu}=0
\]
in the sense of first variation throughout $B_{r_0}(q)$.
Assume that
\[
    \Sigma
    :=
    \spt\|\nu\|
    \cap\bigl(B_{r_0}(q)\setminus\{q\}\bigr)
\]
is a smooth embedded hypersurface satisfying 
\begin{equation}\label{eq:curvature-smallness}
    \lim_{\substack{x\to q\\x\in\Sigma}}
    |x-q|^{j+1}|\nabla^jA_\Sigma(x)|=0
    \quad\text{for every integer }j\geq0.
\end{equation}
Suppose also that
\[
    0<\Theta^n(\|\nu\|,q)<\infty.
\]

Then the following properties hold:
\begin{itemize}
\item [(1)] 
The tangent cone of $\nu$ at $q$ is unique and equals $m_q|P_q|$,
where $P_q$ is an $n$-dimensional linear hyperplane and $m_q=\Theta^n(\|\nu\|,q)$ is a positive integer.

\item [(2)]  There exist $R>r>0$, a unit normal $\eta_q$ to $P_q$,
and smooth functions
\[
    u_\alpha\colon
    B_R^{P_q}\setminus\{0\}
    \to\mathbb R,
    \quad \alpha=1,\ldots,m_q,
\]
 such that
 \[
 u_1\leq u_2\leq\cdots\leq u_{m_q}\quad\text{on } B_R^{P_q}\setminus\{0\},
 \]
\begin{equation}\label{eq:fixed-plane-graph}
    \spt\|\nu\|
    \cap\bigl(B_r(q)\setminus\{q\}\bigr)
    =
    \left(
        q+
        \bigcup_{\alpha=1}^{m_q}
        \graph_{P_q}u_\alpha
    \right)
    \cap\bigl(B_r(q)\setminus\{q\}\bigr).
\end{equation}
\item [(3)]  For every integer $k\geq0$,
\begin{equation}\label{eq:graph-pointwise-decay}
    \lim_{s\downarrow0}
    \max_{1\leq\alpha\leq m_q}
    \sup_{\substack{x\in P_q\\s/2<|x|<2s}}
    s^{k-1}|D^ku_\alpha(x)|=0.
\end{equation}

\end{itemize}
\end{lemma}

\begin{proof}
Set
\[
    \eta_{q,s}(x)=\frac{x-q}{s},\qquad
    \nu_s=(\eta_{q,s})_*\nu,\qquad
    \Sigma_s=\eta_{q,s}(\Sigma).
\]
The self-shrinker equation and monotonicity give
\[
    \|\nu\|(B_s(q))\leq Cs^n,\qquad
    \|\mathbf H_{\nu_s}\|_{L^\infty(B_b)}\leq C_b s
\]
for every fixed $b<\infty$ and all sufficiently small $s$. Thus every
sequence $s_i\downarrow0$ has a subsequence for which $\nu_{s_i}$ converges
to a nonzero stationary integral cone. Moreover, for $0<a<b$ and $j\geq0$,
\[
    \sup_{\Sigma_s\cap A_{a,b}}
    |\nabla^j A_{\Sigma_s}|
    \leq a^{-j-1}
    \sup_{\substack{x\in\Sigma\\0<|x-q|\leq bs}}
    |x-q|^{j+1}|\nabla^j A_\Sigma(x)|
    \to0.
\]
On compact subannuli, the curvature bound gives local graphical disks
of uniform radius, the area bound controls the number of sheets, and
the curvature derivative bounds give smooth convergence of the sheets.

Each limiting sheet is flat. The radial vector is tangent to the
stationary cone almost everywhere, so each limiting planar patch lies
in a hyperplane through the origin. Distinct limiting planes cannot
intersect away from the origin, since the approximating embedded disks
would then intersect transversely. Each planar patch extends to its
entire punctured plane by closedness of the support and local planarity.
Since $n\geq2$, any two distinct hyperplanes through the origin
intersect away from it. Consequently, every tangent cone has the form
\[
    m_q|P|,\qquad m_q=\Theta^n(\|\nu\|,q).
\]

Write $\varrho(x)=|x-q|$. The planar annular limits imply that, as $x\to q$ with $x\in\Sigma$,
\[
    \frac{|(x-q)^{\perp_\Sigma}|}{|x-q|}\to0,
    \qquad
    |\nabla_\Sigma\varrho|^2
    =1-\frac{|(x-q)^{\perp_\Sigma}|^2}{|x-q|^2}
    \to1.
\]
Indeed, a sequence violating the first limit can be rescaled by
$\varrho(x)$ and passed to a smooth planar limit near the unit sphere.
For sufficiently small $b>0$, the flow of
$\nabla_\Sigma\varrho/|\nabla_\Sigma\varrho|^2$ gives, for any $a\in(0,b)$,
\[
    \Sigma\cap B_b(q)
    \cong \bigl(\Sigma\cap\partial B_a(q)\bigr)\times(0,b).
\]
The flow exists between any two positive radii because the intervening
closed annulus in $\Sigma$ is compact. The level set
$\Sigma\cap\partial B_a(q)$ is a compact smooth manifold. Thus there
are finitely many components $\Sigma^\ell$ of $\Sigma\cap B_b(q)$,
each meeting every sphere of radius in $(0,b)$ in a connected set.

Fix a sequence $s_i\downarrow0$ with $\nu_{s_i}\rightharpoonup m_q|P|$.
On any fixed annulus with compact closure away from the origin, the map
\[
    y\longmapsto
    \left(|y|,\frac{\pi_Py}{|\pi_Py|}\right)
\]
is a proper local diffeomorphism on $\Sigma_{s_i}$ for large $i$.
Ordering the heights in its finite fibers gives smooth sheets.
Since $P\cap\mathbb S^n$ is connected, each sheet has a connected
intersection with every sphere in the annulus. The product
decomposition therefore identifies each rescaled $\Sigma^\ell$ with
one sheet. Setting $V_\ell=|\Sigma^\ell|$ in $B_b(q)$, we obtain
\[
    (\eta_{q,s_i})_*V_\ell\rightharpoonup |P|
    \quad\text{on }\mathbb R^{n+1}\setminus\{0\},
    \qquad
    s_i^{-n}\|V_\ell\|(B_{\rho s_i}(q))\leq C\rho^n.
\]
The mass bound extends the convergence through the origin.

To verify the first variation of $V_\ell$ across $q$, take
$X\in C_c^1(B_b(q);\mathbb R^{n+1})$ and a cutoff $\chi_\rho$
that vanishes on $B_\rho(q)$, equals one outside $B_{2\rho}(q)$,
and satisfies $|D\chi_\rho|\leq C/\rho$. Then
\[
\begin{aligned}
    \delta V_\ell(\chi_\rho X)
    &=\frac12\int x^{\perp_{V_\ell}}\cdot\chi_\rho X\,d\|V_\ell\|,\\
    \bigl|\delta V_\ell(X)-\delta V_\ell(\chi_\rho X)\bigr|
    &\leq
    C\bigl(\|DX\|_\infty+\rho^{-1}\|X\|_\infty\bigr)
    \|V_\ell\|(B_{2\rho}(q))
    =O(\rho^{n-1}).
\end{aligned}
\]
Letting $\rho\downarrow0$ gives the self-shrinker equation for
$V_\ell$ throughout $B_b(q)$. In particular,
\[
    (\eta_{q,s_i})_*V_\ell\rightharpoonup |P|,\qquad
    \|\mathbf H_{(\eta_{q,s_i})_*V_\ell}\|_{L^\infty(B_1)}
    \leq Cs_i.
\]
Allard's regularity theorem %\cite[Section~8]{MR307015}
gives a $C^{1,\alpha}$ graph through $q$ with tangent plane $P$.
Elliptic regularity for the self-shrinker graph equation makes this
graph smooth.

Testing the first variation with tangent vector fields gives a
constant positive integer multiplicity $m_\ell$ on each $\Sigma^\ell$.
Thus, locally at $q$,
\[
    \nu=\sum_\ell m_\ell V_\ell,\qquad
    \sum_\ell m_\ell=m_q,\qquad
    \nu_s\rightharpoonup m_q|P| \quad\text{as }s\downarrow0.
\]
This proves {\rm(1)} with $P_q=P$. On a common disk in $P_q$, repeat
each graph according to $m_\ell$ and order the heights. The order is
fixed because distinct graphs are disjoint over the connected
punctured disk. After shrinking the ambient ball, this gives
\eqref{eq:fixed-plane-graph} and proves {\rm(2)}.

Finally, $u^\alpha(0)=0$ and $Du^\alpha(0)=0$. On $s/2<|x|<2s$,
smoothness gives
\[
    s^{-1}|u^\alpha(x)|+|Du^\alpha(x)|\leq Cs,\qquad
    s^{k-1}|D^ku^\alpha(x)|\leq C_k s^{k-1},\quad
    \forall\,k\geq2.
\]
The constants can be chosen uniformly over the finitely many sheets.
Letting $s\downarrow0$ proves \eqref{eq:graph-pointwise-decay}.
\end{proof}

\iffalse

\fi

\bibliographystyle{alpha} % 指定参考文献格式
\bibliography{reference} % 指定 .bib 文件名（不加 .bib 后缀）

\begin{thebibliography}{99}
\bibitem{All72}
W.~K. Allard,
\emph{On the first variation of a varifold},
Ann. of Math. (2) \textbf{95} (1972), no.~3, 417--491.

\bibitem{BK23}
R.~H. Bamler and B. Kleiner,
\emph{On the multiplicity one conjecture for mean curvature flows of surfaces},
arXiv preprint arXiv:2312.02106 (2023).

\bibitem{CY07}
B.-L. Chen and L. Yin,
\emph{Uniqueness and pseudolocality theorems of the mean curvature flow},
Comm. Anal. Geom. \textbf{15} (2007), no.~3, 435--490.

\bibitem{Din18}
Q. Ding,
\emph{A rigidity theorem on the second fundamental form for self-shrinkers},
Trans. Amer. Math. Soc. \textbf{370} (2018), no.~12, 8311--8329.

\bibitem{DX14}
Q. Ding and Y.~L. Xin,
\emph{The rigidity theorems of self-shrinkers},
Trans. Amer. Math. Soc. \textbf{366} (2014), no.~10, 5067--5085.

\bibitem{EM24}
N. Edelen and P. Minter,
\emph{Uniqueness of regular tangent cones for immersed stable hypersurfaces},
Arch. Ration. Mech. Anal. \textbf{248} (2024), Paper No.~121.

\bibitem{Gul76}
R. Gulliver,
\emph{Removability of singular points on surfaces of bounded mean curvature},
J. Differential Geom. \textbf{11} (1976), no.~3, 345--350.

\bibitem{Gun17}
B. G\"uneysu,
\emph{Covariant Schr\"odinger Semigroups on Riemannian Manifolds},
Operator Theory: Advances and Applications, Vol.~264,
Birkh\"auser, Cham, 2017.



\bibitem{GP15} B. G\"uneysu.

\bibitem{Han25}
Y. Han,
\emph{Mean curvature flow with pinched curvature integral},
Differential Geom. Appl. \textbf{99} (2025), Paper No.~102244.

\bibitem{Han26}
Y. Han,
\emph{Well-posedness of mean curvature flow},
arXiv preprint arXiv:2603.22809 (2026).

\bibitem{HP25}
S. Huang and Z. Peng,
\emph{A note on a diffeomorphism criterion via long-time Ricci flow},
arXiv preprint arXiv:2509.05802 (2025).

\bibitem{Hub57}
A. Huber,
\emph{On subharmonic functions and differential geometry in the large},
Comment. Math. Helv. \textbf{32} (1957), 13--72.

\bibitem{Hui84}
G. Huisken,
\emph{Flow by mean curvature of convex surfaces into spheres},
J. Differential Geom. \textbf{20} (1984), no.~1, 237--266.

\bibitem{Hut86}
J.~E. Hutchinson,
\emph{Second fundamental form for varifolds and the existence of surfaces
minimising curvature},
Indiana Univ. Math. J. \textbf{35} (1986), no.~1, 45--71.

\bibitem{Ilm95}
T. Ilmanen,
\emph{Singularities of mean curvature flow of surfaces},
preprint, 1995.

\bibitem{LL25}
M.-C. Lee and T.-K. Lee,
\emph{Local mollification of metrics with small curvature concentration},
arXiv preprint arXiv:2510.12673 (2025).

\bibitem{LW19}
H. Li and B. Wang,
\emph{The extension problem of the mean curvature flow (I)},
Invent. Math. \textbf{218} (2019), no.~3, 721--777.

\bibitem{Mag96}
C. Mantegazza,
\emph{Curvature varifolds with boundary},
J. Differential Geom. \textbf{43} (1996), no.~4, 807--843.

\bibitem{MS73}
J.~H. Michael and L.~M. Simon,
\emph{Sobolev and mean-value inequalities on generalized submanifolds of
$\mathbb{R}^n$},
Comm. Pure Appl. Math. \textbf{26} (1973), no.~3, 361--379.

\bibitem{MS95}
S. M\"uller and V. \v{S}ver\'ak,
\emph{On surfaces of finite total curvature},
J. Differential Geom. \textbf{42} (1995), no.~2, 229--258.

\bibitem{Nak93}
N. Nakauchi,
\emph{A concentration phenomenon in mean curvature flow},
Manuscripta Math. \textbf{78} (1993), 129--147.

\bibitem{Sim93}
L. Simon,
\emph{Existence of surfaces minimizing the Willmore functional},
Comm. Anal. Geom. \textbf{1} (1993), no.~2, 281--326.

\bibitem{Wan20}
B. Wang,
\emph{The local entropy along {R}icci flow---Part B: The pseudolocality
theorems},
arXiv preprint arXiv:2010.09981 (2020).
\end{thebibliography}

\end{document}